\documentclass{article}
\usepackage{amsmath}
\usepackage{amsthm}
\usepackage{amssymb}
\usepackage{amsfonts}
\usepackage{subfig}
\usepackage{graphicx}
\usepackage{amssymb}
\usepackage{authblk}
\usepackage{comment}
\usepackage[mathscr]{euscript}
\usepackage{float}
\usepackage{xcolor}

\newcommand{\RR}{\mathbb{R}}

\newcommand{\ZZ}{\mathbb{Z}}

\renewcommand{\a}{\alpha}

\newcommand{\e}{\varepsilon}

\newcommand{\G}{\Gamma}

\newtheorem{Theorem}{Theorem}[section]

\newtheorem{Corollary}[Theorem]{Corollary}

\newtheorem{Definition}[Theorem]{Definition}

\newtheorem{Proposition}[Theorem]{Proposition}

\newtheorem{Remark}[Theorem]{Remark}

\title{A Generalized Transform Framework for a Nonlinear Model of Cancer Dynamics}
\author{
Gabriela Lopez, Hector Carmenate and Jyrko Correa-Morris\\
Department of Mathematics, Miami Dade College, Miami, FL, USA.\\ 
Gabriela Lopez, gabriela.lopez027@mymdc.net\\ 
Hector Carmenate, hcarmena@mdc.edu\\ 
Jyrko Correa-Morris, jcorrea7@mdc.edu
}

\begin{document}

\maketitle

\abstract{This paper develops a generalized transform framework for a logistic--Allee tumor-growth model. The method combines a generalized Laplace transform, Adomian decomposition, Chebyshev--Pad\'e rational reconstruction, and a \(\mu\)-scaled generalized transform to obtain admissible semi-analytical approximations. Comparisons with experimental tumor-growth data show that the resulting compact representations are stable, admissible, and comparable in error to standard numerical reference solutions.}

\noindent\textbf{Keywords:} generalized fractional derivative; generalized Laplace transform; \(\mu\)-scaled generalized transform; logistic--Allee model; tumor growth; Adomian decomposition; Chebyshev--Pad\'e approximants.

\noindent\textbf{MSC:} 34A08; 44A10; 65L05; 92C50


\section{Introduction}

Generalized transform methods provide a useful framework for constructing semi-analytical approximations to nonlinear differential models. In this work, we combine a generalized Laplace transform, Adomian decomposition, Chebyshev--Pad\'e approximation (GLAP), and a \(\mu\)-scaled generalized transform to study a logistic--Allee tumor-growth model.

Mathematical models of tumor growth connect experimental observations with mechanisms such as cell proliferation, competition, and microenvironmental interaction. Classical growth laws, including exponential, logistic, and Gompertz models, have been widely used to describe tumor dynamics \cite{benzekry2014,laleh2022}. However, more flexible nonlinear models are often needed when growth depends on population density or variable time-scale effects.

One biologically motivated extension is the Allee effect, which describes a positive relationship between population density and per-capita fitness below a critical threshold. In tumor-growth modeling, related mechanisms may arise from autocrine signaling \cite{gerlee2022}, immune evasion, and microenvironmental conditioning. The logistic--Allee model studied by Paix\~{a}o et al.~\cite{paixao2021} captures density-dependent growth limitation using parameters inferred from in vivo data, supporting its relevance in experimental tumor dynamics \cite{neufeld2017,sardar2022}.

Fractional calculus has been successfully applied in several contexts, including tumor-growth modeling and related biological dynamics \cite{valentim2020,martinez2024, monteiro2024}. Khalil et al.~\cite{khalil2014} introduced the conformable fractional derivative, and related local formulations have since been applied to several models \cite{abdeljawad2015,atangana2015,kajouni2021}. In this work, we use the generalized definition introduced in \cite{fleitas2021}. In the range \(0<\alpha\leq 1\) considered here, the operator is interpreted as a generalized local first-order derivative; higher-order cases require a separate treatment.

The main contribution of this work is the development of a generalized transform framework for nonlinear tumor-growth dynamics. More precisely, we construct a generalized Laplace--Adomian--Chebyshev--Padé scheme for the logistic--Allee equation, derive a compact first-order semi-analytical approximation, and introduce a $\mu$-scaled transform that enlarges the admissible transform geometry. The proposed approach is tested against experimental tumor-growth data and standard numerical reference solutions. In addition, the paper establishes admissibility conditions for the transformed approximation, proves the existence of an admissible minimizer, characterizes the admissible region of the parametric kernel family, derives an interior optimality condition for the transform-parameter fitting problem, and proves fixed-point consistency with a first-order truncation estimate.

The paper is organized as follows. 
Section~\ref{sec:theory} presents the preliminary definitions and the main properties of the generalized Laplace transform. 
Section~\ref{sec:method} develops the GLAP formulation for the logistic--Allee tumor-growth model and derives the first-order semi-analytical approximation. 
Section~\ref{sec:optimization_cp} describes the optimization strategy and the Chebyshev--Pad\'e reconstruction for the generalized transform. 
Section~\ref{sec:mu_transform} introduces the \(\mu\)-scaled generalized transform and its operational properties. 
Section~\ref{sec:mu_optimization_cp} describes the optimization strategy and the Chebyshev--Pad\'e reconstruction for the \(\mu\)-scaled generalized transform. 
Sections \ref{sec:inverse_mu} and \ref{sec:residue_formulas} provide the inverse and residue inversion formulas for the generalized and \(\mu\)-scaled transforms. 
Section~\ref{sec:error_admissibility} establishes admissibility, existence of an admissible minimizer, admissible regions for the parametric kernel, interior optimality, fixed-point consistency, and a first-order truncation estimate.
Finally, Section \ref{sec:conclusions} presents the conclusions, and the Appendix contains the technical proofs and operational derivations.

\section{Theoretical Basis}
\label{sec:theory}
The following definition and properties of the generalized fractional derivative can be found in \cite{fleitas2021}.

Given \(r\in\RR\), let \(\lceil r\rceil\) denote the ceiling of \(r\), that is, the smallest integer greater than or equal to \(r\).

\begin{Definition} 
\label{d:g0}
Given an interval $I \subseteq \RR$, $f: I \rightarrow {\mathbb{R}}$, $\alpha \in \RR^+$ and a positive
continuous function $T(t,\alpha )$ on $I$,
the \emph{derivative} $G_{T}^{\alpha }f$ of $f$ of order $\alpha$ at the point $t \in I$ is defined by
\begin{equation*}
G_{T}^{\alpha }f(t)=\lim_{h\to 0} \frac{1}{h^{\lceil \a \rceil}}\sum_{k=0}^{\lceil \a \rceil}(-1)^{k}\binom{\lceil \a \rceil}{k}f\big(t-khT(t,\alpha )\big) .
\end{equation*}
\end{Definition}

If $a=\min\{t \in I\}$ (respectively, $b=\max\{t \in I\}$), then $G_{T}^{\alpha }f(a)$ (respectively, $G_{T}^{\alpha }f(b)$) is defined with
$h\to 0^-$ (respectively, $h\to 0^+$) instead of $h\to 0$ in the limit.

\begin{Theorem} 
\label{t:comp}
Let $I \subseteq \RR$ be an interval, $f: I \rightarrow {\mathbb{R}}$ and $\alpha \in \RR^+$.

$(1)$ If the derivative \(D^{\lceil\a\rceil}f\) exists at the point \(t\in I\), then $f$ is $G_{T}^\a$-differentiable at $t$ and
$G_{T}^{\alpha }f(t)= T(t,\alpha )^{\lceil \a \rceil}D^{\lceil \a \rceil}f(t)$.

$(2)$ If $\a \in (0,1]$, then $f$ is $G_{T}^\a$-differentiable at $t \in I$ if and only if $f$ is differentiable at
$t$;
in this case, we have $G_{T}^{\alpha }f(t)= T(t,\alpha ) f'(t)$.
\end{Theorem}

\begin{Theorem} 
\label{t:prop}
Let $I \subseteq \RR$ be an interval, $f,g: I \rightarrow {\mathbb{R}}$ and $\alpha \in \RR^+$.
Assume that $f,g$ are $G_{T}^\a$-differentiable functions at $t \in I$.
The following statements hold:

$(1)$ $af+b\,g$ is $G_{T}^\a$-differentiable at $t$ for every $a,b \in \mathbb{R}$, and
$G_{T}^{\alpha }(af+b\,g)(t) = a\,G_{T}^{\alpha }f(t)+b\,G_{T}^{\alpha }g(t)$.

$(2)$ If $\alpha \in (0,1]$, then $fg$ is $G_{T}^\a$-differentiable at $t$ and
$G_{T}^{\alpha }(fg)(t)=f(t)G_{T}^{\alpha }g(t)+g(t)G_{T}^{\alpha }f(t)$.

$(3)$ If $\alpha \in (0,1]$ and $g(t) \neq 0$, then $f/g$ is $G_{T}^\a$-differentiable at $t$ and
$G_{T}^{\alpha }(\frac{f}{g})(t)=\frac{g(t)G_{T}^{\alpha }f(t)-f(t)G_{T}^{\alpha }g(t)}{g(t)^{2}}$.

$(4)$ $G_{T}^{\alpha }(\lambda )=0$, for every $\lambda \in \mathbb{R}.$

$(5)$ $G_{T}^{\alpha }(t^{p}) = \frac{\G(p+1)}{\G(p-\lceil\alpha \rceil+1)}t^{p-\lceil\alpha \rceil} T(t,\alpha
)^{{\lceil\alpha \rceil}}$, for every $p\in \mathbb{R} \setminus \ZZ^-\!.$

$(6)$ $G_{T}^{\alpha }(t^{-n}) = (-1)^{\lceil\alpha \rceil}\frac{\G(n+\lceil\alpha \rceil)}{\G(n)}\, t^{-n-\lceil\alpha
\rceil} T(t,\alpha )^{{\lceil\alpha \rceil}}$, for every $n\in \mathbb{Z}^+.$
\end{Theorem}

The following definition of the generalized Laplace transform and some of its properties can be found in \cite{bosch2021}.

We assume that the function $T$ is positive and continuous on $(0,\infty)$, and
satisfies for some $\e > 0$
$$
\int_0^\e \frac{d\omega}{T(\omega,\a)} < \infty,
\qquad
\text{and}
\qquad
\int_0^\infty \frac{d\omega}{T(\omega,\a)} = \infty,
$$
for each $0< \a \le 1$.

Let us define for each $0< \a \le 1$ and $c,t \in \RR$
$$
E_\a (c,t)= \exp \Big( c \int_0^t \frac{d\omega}{T(\omega,\a)} \, \Big).
$$

Note that $E_\a (c,t)$ is an eigenfunction for the operator $G_T^\a$,
since
$$
\begin{aligned}
	G_T^\a (E_\a (c,t))
	& = T(t,\a) \Big( \exp \Big( c \int_0^t \frac{d\omega}{T(\omega,\a)} \, \Big)\Big)'
	\\
	& = T(t,\a) \, \exp \Big( c \int_0^t \frac{d\omega}{T(\omega,\a)} \, \Big) \frac{c}{T(t,\a)}
	= c \, E_\a (c,t).
\end{aligned}
$$

Thus,
$$
\begin{aligned}
	\big(E_\a (c,t)\big)' & = c \, E_\a (c,t) \,\frac1{T(t,\a)} \, .
\end{aligned}
$$

Given $0< \a \le 1$ and a measurable function
$f: [0,\infty) \rightarrow \RR$, we define its {generalized  Laplace  transform} as
$$
\mathcal{L}_T^\a [f](s) = \int_0^\infty E_\a (-s,t) \, f(t) \, \frac{dt}{T(t,\a)}\,,
$$
if $\mathcal{L}_T^\a [\,|f|\,](s) < \infty$,
i.e., $E_\a (-s,t) \, f(t) /T(t,\a) \in L^1([0,\infty))$.
If we consider complex-valued functions instead of real-valued functions, then we can obtain similar results.

\begin{Proposition} \label{p:propL}
	Let $c,k \in \RR$, $0< \a \le 1$ and $f,g: [0,\infty) \rightarrow \RR$ be functions such that
	there exist $\mathcal{L}_T^\a [f](s)$ and $\mathcal{L}_T^\a [g](s)$ for some $s$.
	\begin{enumerate}
	\item[(1)]
	One has
	$$
	\mathcal{L}_T^\a [\,cf+k\,g\,](s) = c \,\mathcal{L}_T^\a [f](s) + k \,\mathcal{L}_T^\a [g](s).
	$$
	
	\item[(2)]
	 If there exists $\mathcal{L}_T^\a \big[E_\a(c,t)f(t)\big](s)$, then
	$$
	\mathcal{L}_T^\a \big[ E_\a(c,t)f(t) \big](s) = \mathcal{L}_T^\a [f](s-c) .
	$$
	\end{enumerate}
\end{Proposition}

\begin{Theorem} \label{t:cv}
	Let $f: [0,\infty) \rightarrow \RR$ be a function such that there exists $\mathcal{L}_T^\a [f](s)$ for some $s$ and $0< \a \le 1$.
	Then
	$$
	\mathcal{L}_T^\a [f(t)](s) = L [f(u(x))](s),
	$$
	where $L$ denotes the usual  Laplace  transform, and $u(x)$ is the inverse function of
	$$
	x (t) = \int_0^t \frac{d\omega}{T(\omega,\a)} \,.
	$$
\end{Theorem}

\begin{Corollary} \label{c:cv}
	Let $f: [0,\infty) \rightarrow \RR$ be a function such that there exists $L [f](s)$ for some $s$ and $0< \a \le 1$.
	The generalized Laplace transform at \(s\) exists for
	$$
	f\Big(\int_0^t \frac{d\omega}{T(\omega,\a)} \,\Big)
	$$
	and
	$$
	\mathcal{L}_T^\a \Big[f\Big(\int_0^t \frac{d\omega}{T(\omega,\a)} \,\Big)\Big](s) = L [f](s).
	$$
\end{Corollary}

\begin{Theorem} \label{t:LG}
	Let $f: [0,\infty) \rightarrow \RR$ be a locally absolutely continuous function such that
	there exist $\mathcal{L}_T^\a [f](s)$ and $\mathcal{L}_T^\a [G_T^\a f](s)$ for some $s$ and $0< \a \le 1$.
	Then
	$$
	\mathcal{L}_T^\a [G_T^\a f](s) = s \, \mathcal{L}_T^\a [f](s) - f(0).
	$$
\end{Theorem}

\section{The GLAP Method}
\label{sec:method}

Paix\~{a}o et al.~\cite{paixao2021} study the following logistic--Allee 
tumor growth model:
\begin{equation}
    N' = aN\left(1 - \frac{N}{b}\right)\left(1 - \frac{c+d}{N+d}\right),
    \quad N(0) = \eta,
\end{equation}
using parameter values inferred from in vivo tumor growth data, where 
$N(t)$ denotes the tumor cell population at time $t$, $a > 0$ is the 
intrinsic growth rate, $b > 0$ is the carrying capacity, and $c, d > 0$ 
are parameters governing the Allee threshold. In vivo data analyzed 
in~\cite{paixao2021} suggested the existence of a weak Allee effect in 
the tumor growth dynamics, supporting the biological relevance of this 
formulation.

In this work, we propose the following fractional-order extension of 
this model, replacing the classical first order derivative with the 
generalized fractional derivative $G_T^\a$ defined 
in Section~\ref{sec:theory}:
\begin{equation} \label{fequat}
\left\{
\begin{aligned}
    G_T^\a N &= aN\left(1 - \frac{N}{b}\right)
    \left(1 - \frac{c+d}{N+d}\right), \qquad \alpha \in (0,1]\\[6pt]
    N(0) &= \eta,
\end{aligned}
\right.
\end{equation}
where the parameters $a$, $b$, $c$, $d$, 
and $\eta$ take the values reported in \cite{paixao2021}. We choose the kernel \(T(t,\alpha)\) so that \(T(t,1)=1\), ensuring that the model reduces to the classical integer-order formulation when \(\alpha=1\).

Equation~\eqref{fequat} can be rewritten in a simpler form. Expanding the product, grouping the rational contributions, and using
\[
\frac{N(t)}{N(t)+d}
=
1-\frac{d}{N(t)+d},
\qquad
\frac{N^2(t)}{N(t)+d}
=
N(t)-d+\frac{d^2}{N(t)+d},
\]
we obtain
\begin{equation}\label{eq2}
G_T^\alpha N(t)
=
k_1N(t)
+
k_2\frac{1}{N(t)+d}
-
k_3N^2(t)
-
k_4,
\end{equation}
where
\[
k_1 = a\left(1+\frac{c+d}{b}\right),
\qquad
k_2 = ad(c+d)\left(1+\frac{d}{b}\right),
\]
\[
k_3 = \frac{a}{b},
\qquad
k_4 = a(c+d)\left(1+\frac{d}{b}\right).
\]

It is worth noting that the algebraic reformulation of the model is particularly convenient for the application of the Adomian Decomposition Method. In this form the equation involves simpler nonlinear operators, $N^2$ and $(N+d)^{-1}$, whose Adomian polynomials admit compact recursive constructions.

In particular, the quadratic term leads to a convolution representation of the polynomials, while the rational term generates a simple recurrence. This structure significantly simplifies the iterative scheme and reduces the computational effort compared with applying ADM directly to the original nonlinear expression.

Applying the generalized Laplace transform $\mathcal{L}_T^{\alpha}$ to equation \eqref{eq2} yields

\begin{equation*}
\mathcal{L}_T^{\alpha}[{ G }_{ T }^{ \alpha } N(t)](s)
=
\mathcal{L}_T^{\alpha}
\left[
k_1N(t)+k_2\frac{1}{N(t)+d}-k_3N^2(t)-k_4
\right](s).
\end{equation*}

By linearity of the transform, this becomes

\begin{align*}
\mathcal{L}_T^{\alpha}[{ G }_{ T }^{ \alpha } N(t)](s)
&=
k_1\,\mathcal{L}_T^{\alpha} [N(t)](s)
+
k_2\,\mathcal{L}_T^{\alpha}\!\left[\frac{1}{N(t)+d}\right](s)
\\&-
k_3\,\mathcal{L}_T^{\alpha}[N^2(t)](s)
-
k_4\,\mathcal{L}_T^{\alpha}[1](s).
\end{align*}

By Theorem~\ref{t:LG} and Proposition $2$, both recalled from \cite{bosch2021}, we have

\begin{align*}
s\mathcal{L}_T^{\alpha}[ N(t)](s)-\eta
&=
k_1\,\mathcal{L}_T^{\alpha} [N(t)](s)
+
k_2\,\mathcal{L}_T^{\alpha}\!\left[\frac{1}{N(t)+d}\right](s)
\\&-
k_3\,\mathcal{L}_T^{\alpha}[N^2(t)](s)
-
\frac{k_4}{s}.
\end{align*}

Solving for \(\mathcal{L}_T^{\alpha}[N(t)](s)\) gives

\begin{equation}\label{l1}
\begin{aligned}
\mathcal{L}_T^{\alpha}[ N(t)](s)
&=
\frac{\eta}{s-k_1} - \frac{k_4}{s(s-k_1)}
+
\frac{k_2}{s-k_1}\,\mathcal{L}_T^{\alpha}\!\left[\frac{1}{N(t)+d}\right](s)
\\
&-
\frac{k_3}{s-k_1}\,\mathcal{L}_T^{\alpha}[N^2(t)](s).
\end{aligned}
\end{equation}

We now assume the series representation
\begin{equation}\label{r1}
N(t)=\sum_{n=0}^{\infty} N_n(t).
\end{equation}
The nonlinear operators are decomposed as follows:

\begin{equation}\label{r2}
f(N)=N^2=\sum_{n=0}^{\infty} B_n(t),
\end{equation}

\begin{equation}\label{r3}
g(N)=\frac{1}{N(t)+d}=\sum_{n=0}^{\infty} A_n(t).
\end{equation}

Substituting \eqref{r1} into the identity

\[
(N(t)+d)\,g(N)=1,
\]

and replacing both functions by their series expansions gives

\[
\left(d+\sum_{n=0}^{\infty}N_n(t)\right)
\left(\sum_{n=0}^{\infty}A_n(t)\right)=1 .
\]

Equating coefficients of equal powers in the resulting Cauchy product
yields the recursive construction of the Adomian polynomials for
\(g(N)=(N+d)^{-1}\). The zeroth term is

\begin{equation}\label{a0}
    A_0=\frac{1}{N_0+d},
\end{equation}

and for \(n\ge1\)

\begin{equation}
    A_n=-\frac{1}{N_0+d}\sum_{k=1}^{n}N_k A_{n-k}.
\end{equation}

This recursive formulation avoids the explicit computation of higher derivatives and eliminates the need to use the general Rach--Adomian formula for the polynomials.\\

For the quadratic nonlinearity $f(N)=N^2$, the resulting
polynomials admit a closed convolution representation obtained
directly from the expansion of $\left(\sum_{k=0}^{\infty}N_k\right)^2$.
This yields

\begin{equation}\label{b0}
    B_n=\sum_{k=0}^{n} N_k N_{n-k}, \qquad n\ge0.
\end{equation}

This expression avoids the explicit computation of higher-order derivatives and leads to a more efficient implementation.\\

Substitution of \eqref{r1}, \eqref{r2}, and \eqref{r3} into \eqref{l1} gives

\begin{align*}
\mathcal{L}_T^{\alpha}\left[ \sum_{n=0}^{\infty} N_n(t)\right](s)
&=
\frac{\eta}{s-k_1} - \frac{k_4}{s(s-k_1)}
+
\frac{k_2}{s-k_1}\,\mathcal{L}_T^{\alpha}\!\left[\sum_{n=0}^{\infty} A_n(t)\right](s)
\\&-
\frac{k_3}{s-k_1}\,\mathcal{L}_T^{\alpha}\!\left[\sum_{n=0}^{\infty} B_n(t)\right](s).
\end{align*}

By linearity of the transform, this yields

\begin{equation}\label{l2}
\begin{aligned}
    \sum_{n=0}^{\infty}\mathcal{L}_T^{\alpha}\left[  N_n(t)\right](s)
    &=
    \frac{\eta}{s-k_1} - \frac{k_4}{s(s-k_1)}
    +
    \frac{k_2}{s-k_1}\,\sum_{n=0}^{\infty}\mathcal{L}_T^{\alpha}\!\left[ A_n(t)\right](s)
    \\&-
    \frac{k_3}{s-k_1}\,\sum_{n=0}^{\infty}\mathcal{L}_T^{\alpha}\!\left[ B_n(t)\right](s).
\end{aligned}
\end{equation}

Matching both sides of \eqref{l2} leads to the following iterative algorithm

$$
\begin{aligned}
\mathcal{L}_T^{\alpha}\left[N_0(t)\right](s)&=\frac{\eta}{s-k_1}- \frac{k_4}{s(s-k_1)},\\
\mathcal{L}_T^{\alpha}\left[N_1(t)\right](s)&=\frac{k_2}{s-k_1}\,\mathcal{L}_T^{\alpha}\!\left[ A_0(t)\right](s)-
\frac{k_3}{s-k_1}\,\mathcal{L}_T^{\alpha}\!\left[ B_0(t)\right](s),\\
\mathcal{L}_T^{\alpha}\left[N_2(t)\right](s)&=\frac{k_2}{s-k_1}\,\mathcal{L}_T^{\alpha}\!\left[ A_1(t)\right](s)-
\frac{k_3}{s-k_1}\,\mathcal{L}_T^{\alpha}\!\left[ B_1(t)\right](s), \\
\quad \vdots & \\
\mathcal{L}_T^{\alpha}\left[N_n(t)\right](s)&=\frac{k_2}{s-k_1}\,\mathcal{L}_T^{\alpha}\!\left[ A_{n-1}(t)\right](s)-\frac{k_3}{s-k_1}\,\mathcal{L}_T^{\alpha}\!\left[ B_{n-1}(t)\right](s), \quad n \geq 1.
\end{aligned}
$$

Solving the first equation of the recursive scheme gives \(N_0\). Substituting \(N_0\) into \eqref{a0} and \eqref{b0} gives \(A_0\) and \(B_0\), which are then used to compute \(N_1\). Higher-order terms can be computed recursively in the same way.\\

Consequently

$$
N_0(t)=\eta\mathcal{L}_{T,\alpha}^{-1}\left[\frac{1}{s-k_1}\right](t)- k_4\mathcal{L}_{T,\alpha}^{-1}\left[\frac{1}{s(s-k_1)}\right](t),
$$

by Proposition $2$ from \cite{bosch2021},

$$
N_0(t)=E_\a (k_1,t)\left(\eta-\frac{k_4}{k_1}\right) + \frac{k_4}{k_1},
$$
and
$$
B_0(t)=N_0^2(t)=\left(E_\a (k_1,t)\left(\eta-\frac{k_4}{k_1}\right) + \frac{k_4}{k_1}\right)^2,
$$

$$
A_0(t)=\frac{1}{N_0(t)+d}=\frac{1}{E_\a (k_1,t)\left(\eta-\frac{k_4}{k_1}\right) + \frac{k_4}{k_1} + d}.
$$

The previous expressions are used to compute \(N_1(t)\):

\begin{equation}\label{N1}
\begin{aligned}
N_1(t)&=\mathcal{L}_{T,\alpha}^{-1}\left[\frac{k_2}{s-k_1}\,\mathcal{L}_T^{\alpha}\!\left[\frac{1}{E_\a (k_1,t)\left(\eta-\frac{k_4}{k_1}\right) + \frac{k_4}{k_1} + d} \right](s)\right](t)\\
&-\mathcal{L}_{T,\alpha}^{-1}\left[\frac{k_3}{s-k_1}\,\mathcal{L}_T^{\alpha}\!\left[\left(E_\a (k_1,t)\left(\eta-\frac{k_4}{k_1}\right) + \frac{k_4}{k_1}\right)^2\right](s)\right](t),
\end{aligned}
\end{equation} 

The following inverse relation follows from Theorem~\ref{t:cv} and Corollary~\ref{c:cv}.

\begin{Corollary}\label{invL}
Let $F(s)$ be a function such that its classical inverse Laplace transform $L^{-1}[F(s)](t)$ exists, and let $0<\alpha \le 1$. Then
\[
\mathcal{L}_{T,\alpha}^{-1}[F(s)](t)
=
L^{-1}[F(s)]\big(x(t)\big),
\qquad
x(t)=\int_0^t \frac{d\omega}{T(\omega,\alpha)}.
\]
\end{Corollary}

\begin{proof}
Let
\[
F(s)=\mathcal{L}_T^\alpha[N](s).
\]
Then, by Theorem~\ref{t:cv}, we have
\[
F(s)=L[N(u(x))](s),
\]
where $u$ is the inverse function of $x(t)$. Applying the classical inverse Laplace transform, we obtain
\[
L^{-1}[F(s)](x)=N(u(x)).
\]
Evaluating at $x=x(t)$, it follows that
\[
L^{-1}[F(s)]\big(x(t)\big)=N(u(x(t)))=N(t).
\]
Since, by definition,
\[
\mathcal{L}_{T,\alpha}^{-1}[F(s)](t)=N(t),
\]
we conclude that
\[
\mathcal{L}_{T,\alpha}^{-1}[F(s)](t)
=
L^{-1}[F(s)]\big(x(t)\big).
\]
\end{proof}

To simplify the notation, let
\[
h_1 := \eta - \frac{k_4}{k_1}, 
\qquad
h_2 := \frac{k_4}{k_1}, 
\qquad
h_3 := \frac{k_4}{k_1} + d.
\]

Applying Corollary~\ref{invL} to \eqref{N1}, we obtain

\[
\begin{aligned}
N_1(t)
&=
k_2\,L^{-1}\!\left[
\frac{1}{s-k_1}L\!\left(\frac{1}{h_1 e^{k_1 x}+h_3}\right)(s)
\right]\!\big(x(t)\big)
\\[1ex]
&\quad
-
k_3\,L^{-1}\!\left[
\frac{1}{s-k_1}L\!\left((h_1 e^{k_1 x}+h_2)^2\right)(s)
\right]\!\big(x(t)\big).
\end{aligned}
\]

Hence

\[
\mathcal{L}_{T,\alpha}^{-1}\!\left[
\frac{1}{s-k_1}\mathcal{L}_T^\alpha[B_0(t)](s)
\right](t)
=
\int_0^{x(t)}
e^{k_1(x(t)-\xi)}
\left(h_1e^{k_1\xi}+h_2\right)^2\,d\xi .
\]

Evaluating the integral, we obtain

\[
\begin{aligned}
\mathcal{L}_{T,\alpha}^{-1}\!\left[
\frac{1}{s-k_1}\mathcal{L}_T^\alpha[B_0(t)](s)
\right](t)
&=
\frac{h_1^2}{k_1}
\left(E_\alpha(2k_1,t)-E_\alpha(k_1,t)\right)
\\
&\quad
+
2h_1h_2\,x(t)\,E_\alpha(k_1,t)
+
\frac{h_2^2}{k_1}
\left(E_\alpha(k_1,t)-1\right).
\end{aligned}
\]

Similarly,

\[
\mathcal{L}_{T,\alpha}^{-1}\!\left[
\frac{1}{s-k_1}\mathcal{L}_T^\alpha[A_0(t)](s)
\right](t)
=
\int_0^{x(t)} e^{k_1(x(t)-\xi)}\frac{1}{h_1e^{k_1\xi}+h_3}\,d\xi.
\]

Evaluating the integral, we obtain
\[
\mathcal{L}_{T,\alpha}^{-1}\!\left[
\frac{1}{s-k_1}\mathcal{L}_T^\alpha[A_0(t)](s)
\right](t)
=
\frac{E_\alpha(k_1,t)-1}{k_1h_3}
+
\frac{h_1E_\alpha(k_1,t)}{k_1h_3^2}
\ln\!\left(
\frac{h_1+h_3E_\alpha(-k_1,t)}{h_1+h_3}
\right).
\]

Combining both contributions,
$$
\begin{aligned}
N_1(t)
&=
k_2\left[
\frac{E_\alpha(k_1,t)-1}{k_1 h_3}
+
\frac{h_1 E_\alpha(k_1,t)}{k_1 h_3^2}
\ln\!\left(
\frac{h_1 + h_3 E_\alpha(-k_1,t)}{h_1 + h_3}
\right)
\right]
\\[1ex]
&\quad
-
k_3\left[
\frac{h_1^2}{k_1}\left(E_\alpha(2k_1,t)-E_\alpha(k_1,t)\right)
+
2h_1 h_2\,x(t)\,E_\alpha(k_1,t)
+
\frac{h_2^2}{k_1}\left(E_\alpha(k_1,t)-1\right)
\right].
\end{aligned}
$$
Although higher-order corrections can be constructed recursively from the same scheme, the first-order approximation already provides a stable representation on the transformed interval considered below.

\section{Optimization Strategy and Chebyshev--Pad\'e Construction}
\label{sec:optimization_cp}

In this work, the kernel \(T\) is treated as an optimization-controlled deformation kernel for the transformed representation rather than as a prescribed function chosen a priori. We considered the parametric family
\[
T(t,\alpha)=\sigma^{1-\alpha}\left(1+c(1-\alpha)t\right)^p,
\]
which, for \(p\neq 1\), gives
\[
x(t;\alpha,c,\sigma,p)
=
\sigma^{\alpha-1}
\frac{
\left(1+c(1-\alpha)t\right)^{1-p}-1
}{
c(1-\alpha)(1-p)
}.
\]
This family is consistent with the classical case, since \(T(t,1)=1\).

The transformation parameters were computed by optimizing the first-order approximation
\[
S_1(x(t;z))=N_0(x(t;z))+N_1(x(t;z)),
\qquad
z=(\alpha,c,\sigma,p),
\]
against the experimental data. The objective function was
\[
\mathcal{J}(z)
=
\left[
\frac{1}{m}
\sum_{i=1}^{m}
\left(
S_1(x(t_i;z))-N_i^{\mathrm{data}}
\right)^2
\right]^{1/2}.
\]
The admissible candidates were required to satisfy
\[
1+c(1-\alpha)t>0
\]
and
\[
\frac{h_1+h_3e^{-k_1x(t;z)}}{h_1+h_3}>0
\]
on \(0\leq t\leq 120\). These conditions keep the transformed variable well defined and ensure that the logarithmic term in \(N_1\) remains real-valued. The interval slightly exceeds the last experimental time \(t=114.039\), giving a safety margin beyond the data range.

The critical point associated with the logarithmic term is determined by
\[
h_1+h_3e^{-k_1x}=0,
\qquad
x_{\mathrm{crit}}
=
-\frac{1}{k_1}
\ln\left(-\frac{h_1}{h_3}\right).
\]
For \(\eta=271.309\), this gives
\[
x_{\mathrm{crit}}\approx 4.70388.
\]
Since
\[
x(114.039)\approx 3.95690,
\]
the optimized transformation keeps the approximation inside the admissible real domain over the data range.

The optimized parameters were
\[
\alpha=0.621897419143,\qquad
c=6.098704653588\times 10^{-7},
\]
\[
\sigma=1600.310101849,\qquad
p=48562.047025080.
\]

These values are transform parameters, not biological parameters of the tumor-growth model. Their role is to define an admissible transformed representation where the semi-analytical approximation remains real-valued and stable on the data interval. In the optimization process, admissible candidates with different parameter magnitudes produced reasonable fits, suggesting a nonconvex landscape with possible parameter degeneracy.

The reported values correspond to the lowest-error admissible parameter set identified during the optimization process. The admissible transform-parameter region and the associated interior optimality condition are analyzed in Section~\ref{sec:error_admissibility}.

A multistage global--local optimization strategy was used. The admissible range of \(p\) was divided into several bands. In each band, candidates were generated by means of global log-uniform sampling and local sampling around the best admissible basin.

Particle swarm optimization was then employed to explore each band, and the best candidates were refined with \texttt{fmincon} in MATLAB R2023B using both SQP and interior-point algorithms.

Before applying the rational correction, the raw approximation \(S_1(x(t))\) was examined. Figure~\ref{fig:S1-raw} shows that \(S_1\) is smooth and stable over the transformed data interval. This is relevant for the present equation: in the classical formulation, the corresponding \(N_1\) term reaches the boundary of its real admissible domain before covering the full data range.

\begin{figure}[htbp]
    \centering
    \includegraphics[
        width=0.85\textwidth,
        height=7cm,
        keepaspectratio
    ]{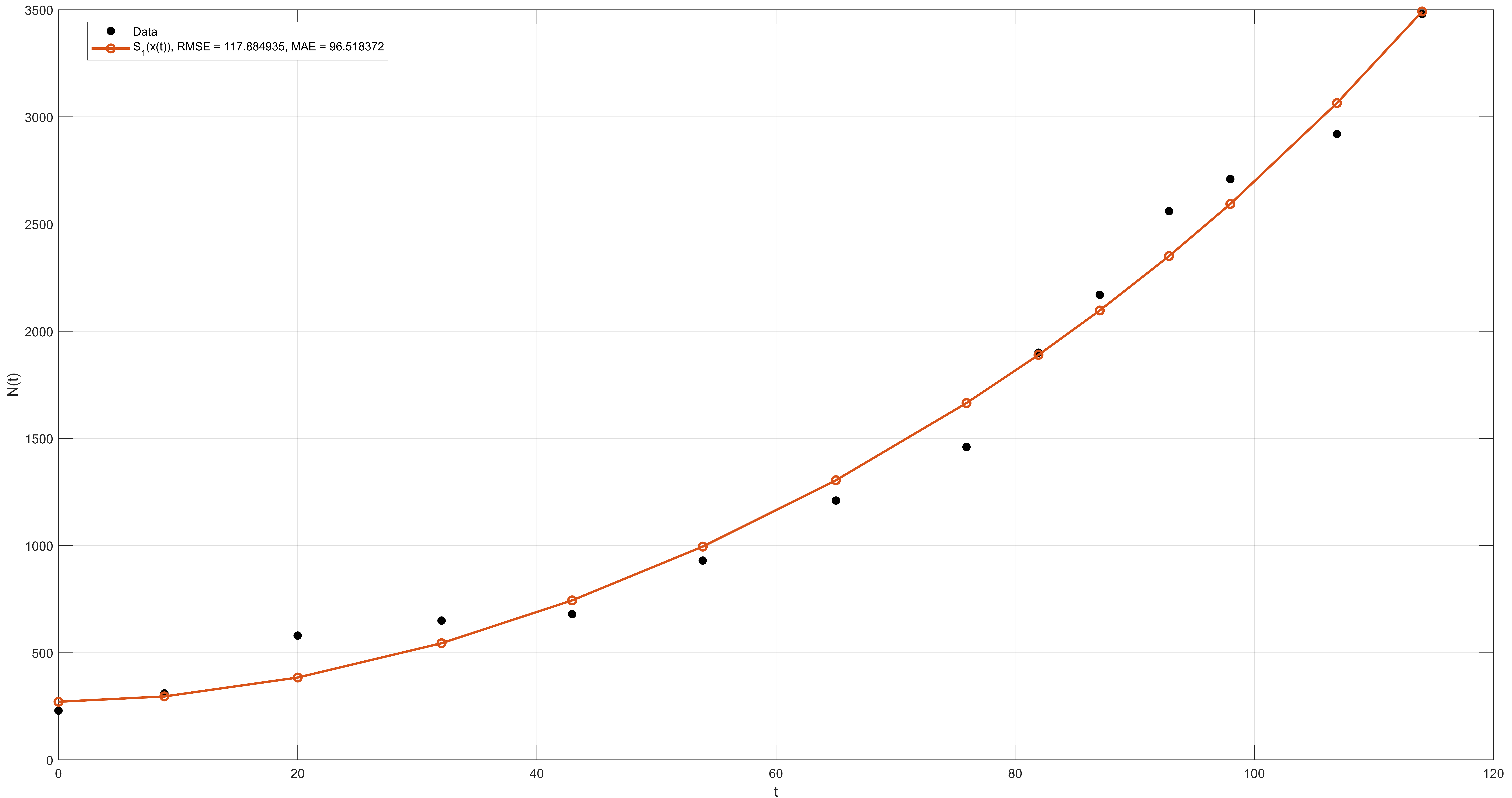}
    \caption{First-order approximation \(S_1(x(t))\) before the Chebyshev--Pad\'e reconstruction 
(RMSE = 117.88, MAE = 96.52, MAX = 209.42).}
    \label{fig:S1-raw}
\end{figure}

After the transformation was fixed, Maple 2026 was used to construct a Chebyshev--Pad\'e approximation directly from the analytical expression
\[
S_1(x)=N_0(x)+N_1(x).
\]
The approximation was built in the variable \(x\) and then evaluated. Maple's \texttt{chebpade} command was applied on an interval \([0,x_{\max}]\). The transformed data range was
\[
x(114.039)\approx 3.95690,
\]
and a slight refinement of the construction endpoint was tested. The best result was obtained with
\[
x_{\max}^{\mathrm{Cheb}}=3.991,\qquad [L/M]=[1/2].
\]

The corresponding rational approximation derived by Maple is reported below, with coefficients rounded for readability:
\[
R_{1,2}(x)
=
-\frac{-595.663290713065+40.791184534975\,x}
{2.346174259049-1.027911760024\,x+0.118034500717\,x^2}.
\]

A key advantage of the present approach is that the optimized \(S_1(x)\) is already well behaved on the relevant transformed interval. Thus, the rational correction can be constructed from a Chebyshev representation rather than from a purely local Taylor-type basis. 

In many Laplace--Adomian implementations, Pad\'e approximants are used to extend truncated local expansions. Here, the generalized transformation first produces an admissible low-order representation, and the rational reconstruction is then applied to that representation.

This admissibility is not obtained in the corresponding classical formulation for the present equation: the first correction term reaches the boundary of its real domain before covering the full data range. The generalized transformed formulation avoids this obstruction and gives a compact first-order approximation on the complete interval.

For reference, we reconstructed numerical solutions using MATLAB's \texttt{ode45} solver and a fixed-step RK4 scheme, following Paix\~{a}o et al.~\cite{paixao2021}. Table~\ref{T1} and Figure~\ref{fig:real-chebpade-ode45-rk4} show that the Chebyshev--Pad\'e reconstruction gives comparable error levels while retaining an explicit semi-analytical structure.
\begin{figure}[H]
    \centering
    \includegraphics[
        width=0.95\textwidth,
        height=8cm,
        keepaspectratio
    ]{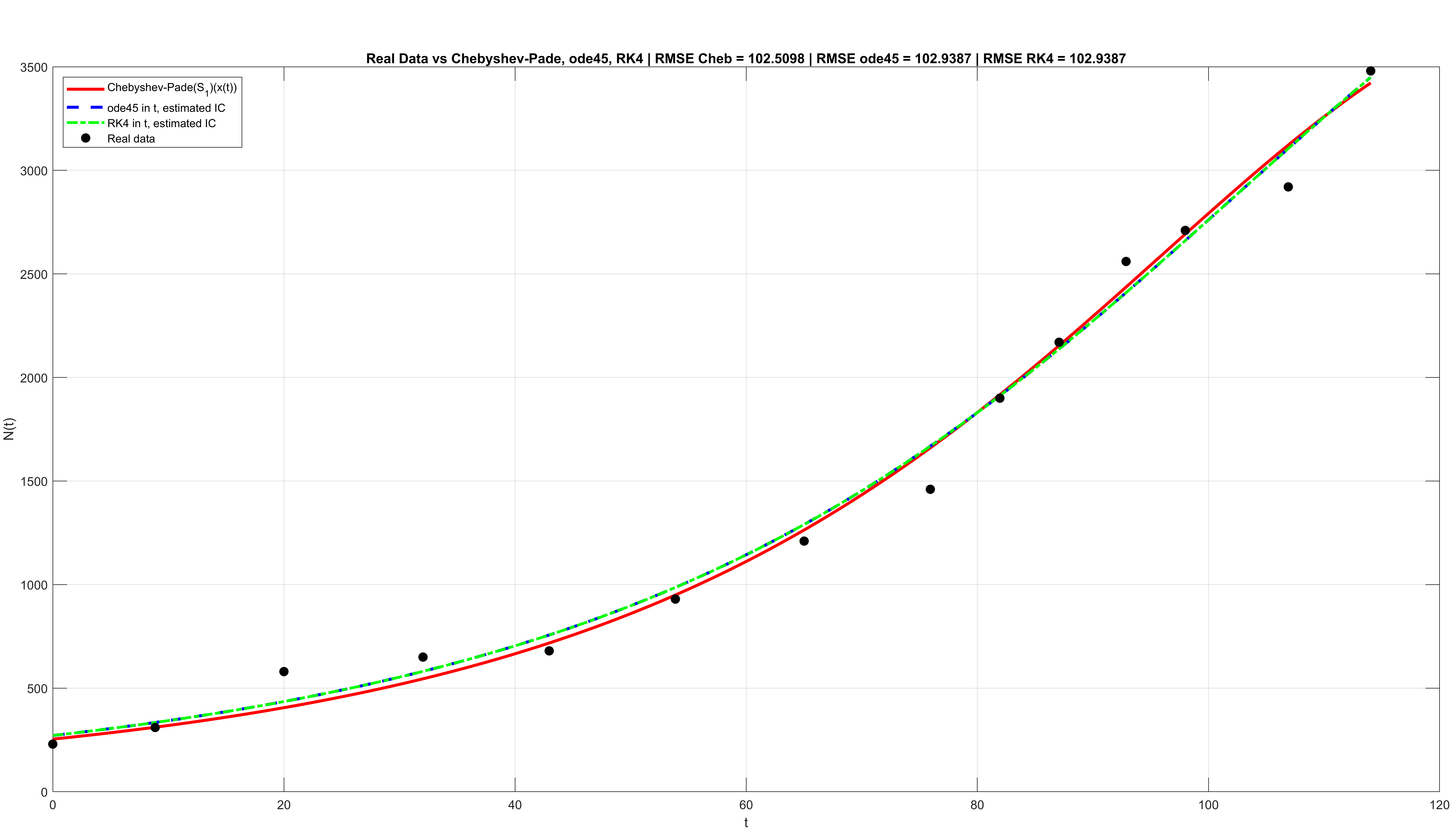}
    \caption{Real data versus Chebyshev--Pad\'e approximation, \texttt{ode45}, and RK4.}
    \label{fig:real-chebpade-ode45-rk4}
\end{figure}

\begin{table}[H]
\centering
\caption{Error comparison for the Chebyshev--Pad\'e approximation obtained from \(S_1(x(t))\).}
\label{T1}
\begin{tabular}{|c|c|c|c|}
\hline
\textbf{Method} & \textbf{RMSE} & \textbf{MAE} & \textbf{MAX} \\
\hline
\(\text{Cheb-Pad\'e}(S_1)(x(t))\)
& 102.51
& 75.0
& 202.3 \\
\hline
\(\texttt{ode45}\ \text{estIC}\)
& 102.94
& 82.9
& 208.4 \\
\hline
\(\text{RK4}\ \text{estIC}\)
& 102.94
& 82.9
& 208.4 \\
\hline
\end{tabular}
\end{table}

\section{The $\mu$-Scaled Generalized Transform}
\label{sec:mu_transform}

The generalized Laplace transform deforms the independent variable through $x(t)=\int_0^t \frac{d\omega}{T(\omega,\alpha)}.$
However, the admissibility and accuracy of the Laplace--Adomian approximation may also depend on the scale of the transformed interval. This motivates the \(\mu\)-scaled transformed variable $y_\mu(t)=\mu x(t/\mu),$ together with the corresponding scaled kernel \(T(t/\mu,\alpha)\).

The parameter \(\mu>0\) generates a family of transformed representations associated with the same model. Unlike an external rescaling after the approximation is constructed, \(\mu\) remains part of the inverse representation and affects the domain where the semi-analytical approximation is evaluated. This additional degree of freedom is useful in the present problem because the first correction term \(N_1\) contains a logarithmic factor whose real-valued admissibility depends on the transformed variable.

We develop below the basic properties needed for the Transform--Adomian--Pad\'e construction.

The following defines the $\mu-$scaled generalized derivative, which is a generalization of the generalized derivative given in Definition \ref{d:g0}.

\begin{Definition}
\label{NGD}
Given an interval $I \subseteq \RR$, $f: I \rightarrow {\mathbb{R}}$, $\alpha,\mu \in \RR^+$, such that $\mu \in (0,\infty)$, and a positive
continuous function $T(\frac{t}{\mu},\alpha )$ on $I$,
the \emph{derivative} $G_{T,\mu}^{\alpha }f$ of $f$ of order $\alpha$ at the point $t \in I$ is defined by
\begin{equation*}
G_{T,\mu}^{\alpha}f(t)
=
\lim_{h\to 0}
\frac{1}{h^{{\lceil \alpha \rceil}}}
\sum_{k=0}^{{\lceil \alpha \rceil}}
(-1)^k
\binom{{\lceil \alpha \rceil}}{k}
f\!\left(t-kh\,T\!\left(\frac{t}{\mu},\alpha\right)\right).
\end{equation*}
\end{Definition}

\begin{Theorem}
\label{t:comp2}
Let $I \subseteq \RR$ be an interval, $f: I \rightarrow {\mathbb{R}}$ and $\alpha,\mu \in \RR^+$.

$(1)$ If the derivative \(D^{\lceil\a\rceil}f\) exists at the point \(t\in I\), then $f$ is $G_{T,\mu}^\a$-differentiable at $t$ and
$G_{T,\mu}^{\alpha }f(t)= T(\frac{t}{\mu},\alpha)^{\lceil \a \rceil}D^{\lceil \a \rceil}f(t)$.

$(2)$ If $\a \in (0,1]$, then $f$ is $G_{T,\mu}^\a$-differentiable at $t \in I$ if and only if $f$ is differentiable at
$t$;
in this case, we have $G_{T,\mu}^{\alpha }f(t)= T(\frac{t}{\mu},\alpha ) f'(t)$.
\end{Theorem}

The proof is given in Appendix~\ref{app:proof_tcomp2}.

For each $0< \a \le 1$, and $c,t \in \RR$, $\mu \in \RR^+$ where $\mu \in (0,\infty)$, we define

$$
E_{\alpha}(\mu c,t/\mu)
=
\exp\left(
\mu c\int_{0}^{t/\mu}\frac{d\omega}{T(\omega,\alpha)}
\right).
$$
A direct computation gives

\begin{equation*}
    \begin{aligned}
     G_{T,\mu}^{\alpha}\left[E_{\alpha}(\mu c,t/\mu)\right]
    &=
    T\!\left(\frac{t}{\mu},\alpha\right)\,
    \exp\left(
    \mu c\int_{0}^{t/\mu}\frac{d\omega}{T(\omega,\alpha)}
    \right)'\\
    &=
    T\!\left(\frac{t}{\mu},\alpha\right)
    \exp\left(
    \mu c\int_{0}^{t/\mu}\frac{d\omega}{T(\omega,\alpha)}
    \right)\,
    \frac{\mu c}{\mu\,T(t/\mu,\alpha)}\\
    &=
    c\, E_{\alpha}(\mu c,t/\mu).
    \end{aligned}
\end{equation*}

This proves that $E_{\alpha}(\mu c,t/\mu)$ is an eigenfunction.

\begin{Definition}
Given \(0<\alpha\leq 1\), \(0<\mu<\infty\), and a measurable function
\(f:[0,\infty)\to\RR\), we define its \(\mu\)-scaled generalized transform as
\begin{equation*}
CC_{T,\mu}^{\alpha}[f](s,\mu)
=
\mu\int_{0}^{\infty}
e^{-s\mu x(t/\mu)}\,f(t)\,
\frac{dt}{T(t/\mu,\alpha)},
\end{equation*}
whenever
\[
\mu\int_{0}^{\infty}
\left|e^{-s\mu x(t/\mu)}f(t)\right|
\frac{dt}{T(t/\mu,\alpha)}
<\infty.
\]
\end{Definition}

Here, we assume the same conditions on the functions \(T\) and \(x(t)\) as in the generalized Laplace transform. Note that

$$
x\!\left(\frac{t}{\mu}\right)=\int_{0}^{t/\mu}\frac{d\omega}{T(\omega,\alpha)}, \qquad \left(x\!\left(\frac{t}{\mu}\right)\right)'=\frac{1}{\mu\,T(t/\mu,\alpha)}.
$$

Moreover, the \(\mu\)-scaled transform recovers the generalized Laplace transform when \(\mu=1\):
$$
CC_{T,1}^{\alpha}[f](s,1)=\mathcal{L}_{T}^{\alpha}[f](s).
$$

\begin{Theorem}\label{t:mu_exponential_order}
Let \(f:[0,\infty)\to\RR\) be a measurable and locally integrable function, and let
\(0<\alpha\leq 1\), \(0<\mu<\infty\). Suppose that there exist constants
\(M>0\), \(c\in\RR\), and \(t_0\geq0\) such that
\[
|f(t)|\leq M E_{\alpha}(\mu c,t/\mu),
\qquad t\geq t_0.
\]
Then \(CC_{T,\mu}^{\alpha}[f](s,\mu)\) exists for
\(\Re(s)>c\). Moreover,
\[
\lim_{t\to\infty}
e^{-s\mu x(t/\mu)}f(t)=0,
\qquad \Re(s)>c.
\]
If the above estimate holds for every \(t\geq0\), then
\[
\left|
CC_{T,\mu}^{\alpha}[f](s,\mu)
\right|
\leq
\frac{\mu M}{\Re(s)-c},
\qquad \Re(s)>c.
\]
\end{Theorem}

The proof is given in Appendix~\ref{app:proof_mu_exponential_order}.

\begin{Theorem}
\label{t:mu_derivative_transform}
Let \(f:[0,\infty)\to\RR\) be a locally absolutely continuous function such that
\(CC_{T,\mu}^{\alpha}[f](s,\mu)\) and
\(CC_{T,\mu}^{\alpha}[G_{T,\mu}^{\alpha}f](s,\mu)\) exist for some \(s\), with
\(0<\alpha\leq 1\) and \(0<\mu<\infty\). Then
\begin{equation}
CC_{T,\mu}^{\alpha}[G_{T,\mu}^{\alpha}f](s,\mu)
=
s\,CC_{T,\mu}^{\alpha}[f](s,\mu)-\mu f(0).
\end{equation}
\end{Theorem}

The proof is given in Appendix~\ref{app:proof_mu_derivative_transform}.

\begin{Proposition}\label{constant}
There exists the \(\mu\)-scaled generalized transform of the constant function \(1\) for \(0<\alpha\leq 1\) and \(0<\mu<\infty\). If \(s>0\), then
$$
CC_{T,\mu}^{\alpha}[1](s,\mu)=\frac{\mu}{s}.
$$
\end{Proposition}
The proof is given in Appendix~\ref{app:proof_constant}.

\begin{Proposition}\label{exp}
For \(0<\alpha\leq 1\), \(0<\mu<\infty\), and \(c\in\mathbb{C}\), one has
$$
CC_{T,\mu}^{\alpha}\left[E_\alpha(\mu c,t/\mu)\right](s,\mu)
=
\frac{\mu}{s-c},
$$
whenever the integral exists.
\end{Proposition}

The proof is given in Appendix~\ref{app:proof_exp}.

\begin{Theorem}
\label{t:mu_laplace_relation}
Let \(f:[0,\infty)\to\RR\) be a measurable function such that
\(CC_{T,\mu}^{\alpha}[f](s,\mu)\) exists for some \(s\), with
\(0<\alpha\leq 1\) and \(0<\mu<\infty\). Let
$$
y_\mu(t)=\mu x\left(\frac{t}{\mu}\right),
\qquad
x(t)=\int_0^t\frac{d\omega}{T(\omega,\alpha)}.
$$
If \(u_\mu=y_\mu^{-1}\) and
$$
g_\mu(r)=f(u_\mu(r)),
$$
then
$$
CC_{T,\mu}^{\alpha}[f](s,\mu)
=
\mu L[g_\mu](s).
$$
\end{Theorem}

The proof is given in Appendix~\ref{app:proof_mu_laplace_relation}.

\begin{Corollary}\label{invrelat}
Under the hypotheses of the previous theorem, if
\(F(s,\mu)=CC_{T,\mu}^{\alpha}[f](s,\mu)\), then
$$
f(t)
=
L^{-1}\left[
\frac{1}{\mu}F(s,\mu)
\right](y_\mu(t)).
$$
Equivalently,
$$
f(t)
=
L^{-1}\left[
\frac{1}{\mu}CC_{T,\mu}^{\alpha}[f](s,\mu)
\right]
\left(
\mu x\left(\frac{t}{\mu}\right)
\right).
$$
\end{Corollary}
The proof is given in Appendix~\ref{app:proof_invrelat}.

\begin{Theorem}
\label{t:mu_convolution}
Let \(f,g:[0,\infty)\to\RR\) be measurable functions such that all the transforms involved are well defined, with \(0<\alpha\leq 1\) and \(0<\mu<\infty\). Let
$$
y_\mu(t)=\mu x\left(\frac{t}{\mu}\right),
\qquad
x(t)=\int_0^t\frac{d\omega}{T(\omega,\alpha)},
\qquad
u_\mu=y_\mu^{-1}.
$$
Define the \(\mu\)-convolution by
$$
(f*_\mu g)(t)
=
\int_0^t
f\!\left(u_\mu(y_\mu(t)-y_\mu(\omega))\right)
g(\omega)
\frac{d\omega}{T(\omega/\mu,\alpha)}.
$$
Then
$$
CC_{T,\mu}^{\alpha}[f*_\mu g](s,\mu)
=
\frac{1}{\mu}
CC_{T,\mu}^{\alpha}[f](s,\mu)\,
CC_{T,\mu}^{\alpha}[g](s,\mu).
$$
\end{Theorem}

The proof is given in Appendix~\ref{app:proof_mu_convolution}.

The first two components of the Transform--Adomian decomposition are produced as follows.

Applying the \(\mu\)-scaled generalized transform to the linear part of the
recursive scheme gives
$$
s\,CC_{T,\mu}^{\alpha}[N](s,\mu)-\mu\eta
=
k_1\,CC_{T,\mu}^{\alpha}[N](s,\mu)
-
k_4\,CC_{T,\mu}^{\alpha}[1](s,\mu).
$$

The zeroth component satisfies
$$
CC_{T,\mu}^{\alpha}[N_0](s,\mu)
=
\frac{\mu\eta}{s-k_1}
-
\frac{k_4\mu}{s(s-k_1)}.
$$
Using Corollary~\ref{invrelat} and Propositions~\ref{constant} and~\ref{exp}, we obtain
$$
N_0(t,\mu)
=
E_{\alpha}(\mu k_1,t/\mu)
\left(
\eta-\frac{k_4}{k_1}
\right)
+
\frac{k_4}{k_1}.
$$
Since \(x(0)=0\), it follows that
$$
N_0(0,\mu)=\eta,
\qquad 0<\mu<\infty,
$$
so the initial condition is preserved for every finite positive value of \(\mu\).

For convenience, define
$$
h_1:=\eta-\frac{k_4}{k_1},
\qquad
h_2:=\frac{k_4}{k_1},
\qquad
h_3:=\frac{k_4}{k_1}+d.
$$

The first correction component \(N_1(t,\mu)\) satisfies
$$
CC_{T,\mu}^{\alpha}[N_1](s,\mu)
=
\frac{k_2}{s-k_1}CC_{T,\mu}^{\alpha}[A_0](s,\mu)
-
\frac{k_3}{s-k_1}CC_{T,\mu}^{\alpha}[B_0](s,\mu).
$$
Using the \(\mu\)-convolution theorem, we obtain
$$
N_1(t,\mu)
=
E_{\alpha}(\mu k_1,t/\mu)
\left[
k_2I_1(t,\mu)-k_3I_2(t,\mu)
\right],
$$
where
$$
I_1(t,\mu)
=
\frac{1}{k_1h_3}\left(1-E_{\alpha}(-\mu k_1,t/\mu)\right)
+
\frac{h_1}{k_1h_3^2}
\ln\!\left(
\frac{h_1E_{\alpha}(\mu k_1,t/\mu)+h_3}
{(h_1+h_3)E_{\alpha}(\mu k_1,t/\mu)}
\right),
$$
and
$$
I_2(t,\mu)
=
\frac{h_1^2}{k_1}\left(E_{\alpha}(\mu k_1,t/\mu)-1\right)
+
2\mu h_1h_2\,x(t/\mu)
+
\frac{h_2^2}{k_1}
\left(1-E_{\alpha}(-\mu k_1,t/\mu)\right).
$$
The detailed computation of \(I_1(t,\mu)\) and \(I_2(t,\mu)\) is given in Appendix~\ref{app:computation_N1_mu}.

Equivalently, replacing \(h_1,h_2,h_3\) by their definitions yields
$$
\begin{aligned}
N_1(t,\mu)
&=
E_{\alpha}(\mu k_1,t/\mu)
\Bigg[
k_2\Bigg(
\frac{1}{k_1\left(\frac{k_4}{k_1}+d\right)}
\left(1-E_{\alpha}(-\mu k_1,t/\mu)\right)
\\
&\qquad\qquad
+
\frac{\eta-\frac{k_4}{k_1}}
{k_1\left(\frac{k_4}{k_1}+d\right)^2}
\ln\!\left(
\frac{
\left(\eta-\frac{k_4}{k_1}\right)E_{\alpha}(\mu k_1,t/\mu)
+
\left(\frac{k_4}{k_1}+d\right)
}{
(\eta+d)E_{\alpha}(\mu k_1,t/\mu)
}
\right)
\Bigg)
\\
&\qquad
-
k_3\Bigg(
\frac{\left(\eta-\frac{k_4}{k_1}\right)^2}{k_1}
\left(E_{\alpha}(\mu k_1,t/\mu)-1\right)
+
2\mu\left(\eta-\frac{k_4}{k_1}\right)\frac{k_4}{k_1}x(t/\mu)
\\
&\qquad\qquad
+
\frac{\left(\frac{k_4}{k_1}\right)^2}{k_1}
\left(1-E_{\alpha}(-\mu k_1,t/\mu)\right)
\Bigg)
\Bigg].
\end{aligned}
$$

Since
$$
(\eta+d)E_{\alpha}(\mu k_1,t/\mu)>0,
$$
the real-valuedness of the logarithmic term is determined by the numerator. Therefore, the formula for \(N_1(t,\mu)\) requires
$$
\left(\eta-\frac{k_4}{k_1}\right)E_{\alpha}(\mu k_1,t/\mu)
+
\left(\frac{k_4}{k_1}+d\right)>0.
$$
\section{Optimization Strategy for the \(\mu\)-Scaled Generalized Transform and Chebyshev--Pad\'e Construction}
\label{sec:mu_optimization_cp}

The previous construction is extended by introducing the \(\mu\)-scaled generalized transform, whose transformed variable is
\[
y_\mu(t)=\mu x\left(\frac{t}{\mu}\right),
\qquad
x(t)=\int_0^t \frac{d\omega}{T(\omega,\alpha)}.
\]
This extension generates a family of transformed variables indexed by \(0<\mu<\infty\). Unlike a simple external rescaling, the parameter \(\mu\) remains present after applying the inverse transform, and therefore different values of \(\mu\) lead to different semi-analytical representations associated with the same underlying model. To keep the differential dynamics consistent with this transformed variable, the generalized derivative is evaluated through the scaled kernel \(T(t/\mu,\alpha)\).

For the same parametric family
\[
T(t,\alpha)= \sigma^{1-\alpha}\left(1+c(1-\alpha)t\right)^p,
\]
the transformed variable associated with the \(\mu\)-scaled construction is
\[
y_\mu(t;\alpha,c,\sigma,p)
=
\mu  \sigma^{\alpha-1}
\frac{
\left(1+c(1-\alpha)\frac{t}{\mu}\right)^{1-p}-1
}{
c(1-\alpha)(1-p)
},
\qquad p\neq 1.
\]
The case \(\mu=1\) recovers the transformed variable of the generalized Laplace transform. Thus, the parameter \(\mu\) introduces an additional degree of freedom in the analytical representation.

The optimization was performed on the first-order approximation
\[
S_1(y_\mu(t;z,\mu))=N_0(y_\mu(t;z,\mu))+N_1(y_\mu(t;z,\mu)),
\qquad
z=(\alpha,c, \sigma,p),
\]
where \(\mu\) was fixed during each internal optimization. The objective function was
\[
\mathcal{J}(z,\mu)
=
\left[
\frac{1}{m}
\sum_{i=1}^{m}
\left(
S_1(y_\mu(t_i;z,\mu))-N_i^{\mathrm{data}}
\right)^2
\right]^{1/2}.
\]
The optimization is performed over the transformation parameters, while the fitted quantity remains the analytical approximation \(S_1(y_\mu(t))\).

The admissible candidates were required to satisfy
\[
1+c(1-\alpha)\frac{t}{\mu}>0
\]
and
\[
\frac{h_1+h_3e^{-k_1y_\mu(t;z,\mu)}}{h_1+h_3}>0
\]
on the prescribed interval \(0\leq t\leq 120\). These conditions keep the transformed variable well defined and ensure that the logarithmic term appearing in \(N_1\) remains real-valued. As in the generalized Laplace transform case, this interval slightly exceeds the last experimental time \(t=114.039\), providing a safety margin beyond the data range.

The critical point is determined by the zero of the logarithmic argument,
\[
h_1+h_3e^{-k_1y}=0.
\]
Therefore,
\[
y_{\mathrm{crit}}
=
-\frac{1}{k_1}
\ln\left(-\frac{h_1}{h_3}\right).
\]
For the estimated initial condition \(\eta=271.309\), this gives
\[
y_{\mathrm{crit}}\approx 4.70388.
\]
For the optimized \(\mu\)-scaled transform, the largest transformed value on the interval of interest satisfies
\[
y_\mu(120)\approx 4.05571,
\]
and the largest transformed data point satisfies
\[
y_\mu(114.039)\approx 3.95690.
\]
Thus, the optimized representation remains inside the admissible real domain of \(N_1\) over the full interval considered.

The optimized parameters for the \(\mu\)-scaled generalized transform were
\[
\mu=9.5\times 10^{5},\qquad
\alpha=0.627828,\qquad
c=0.00320005,
\]
\[
\sigma=1800.00898,\qquad
p=8.93219\times 10^{6}.
\]

The same interpretation extends naturally to the $\mu$-scaled formulation. The parameter \(\mu\), together with \(\alpha,c, \sigma,\) and \(p\), defines a transformed representation and should not be assigned direct biological meaning.

Large values of $\mu$ and $p$ should be interpreted as features of the transformed representation selected by the admissibility-constrained optimization, rather than as physical characteristics of the underlying tumor-growth process.

Smaller parameter regimes also produced admissible approximations during the search, but the larger-scale representation reported here gave the lowest error, although the improvement was modest.

This suggests a nonconvex optimization landscape in which different transform geometries may lead to comparable admissible fits. 
The corresponding parameter region and interior optimality condition are characterized in Section~\ref{sec:error_admissibility}.

These parameters were generated through a multistage global--local optimization strategy. A range of fixed \(\mu\)-values was tested, and for each fixed \(\mu\) the remaining parameters \((\alpha,c, \sigma,p)\) were optimized. The admissible range of \(p\) was divided into several bands. In each band, candidates were generated using global log-uniform sampling and local sampling around the best admissible basins.

A particle swarm search was then performed, and the best candidates were refined using constrained local optimization with \texttt{fmincon}. Final refinements were carried out using both SQP and interior-point algorithms, and the parameter set with the lowest value of \(\mathcal{J}(z,\mu)\) was selected.

Before applying the rational correction, the raw approximation \(S_1(y_\mu(t))\) was examined. Figure~\ref{fig:S1-raw-mu} shows that this first-order approximation is smooth and admissible over the data interval. This confirms that the \(\mu\)-scaled transform provides a stable low-order representation before the Chebyshev--Pad\'e reconstruction is applied.

\begin{figure}[H]
    \centering
    \includegraphics[
        width=0.85\textwidth,
        height=7cm,
        keepaspectratio
    ]{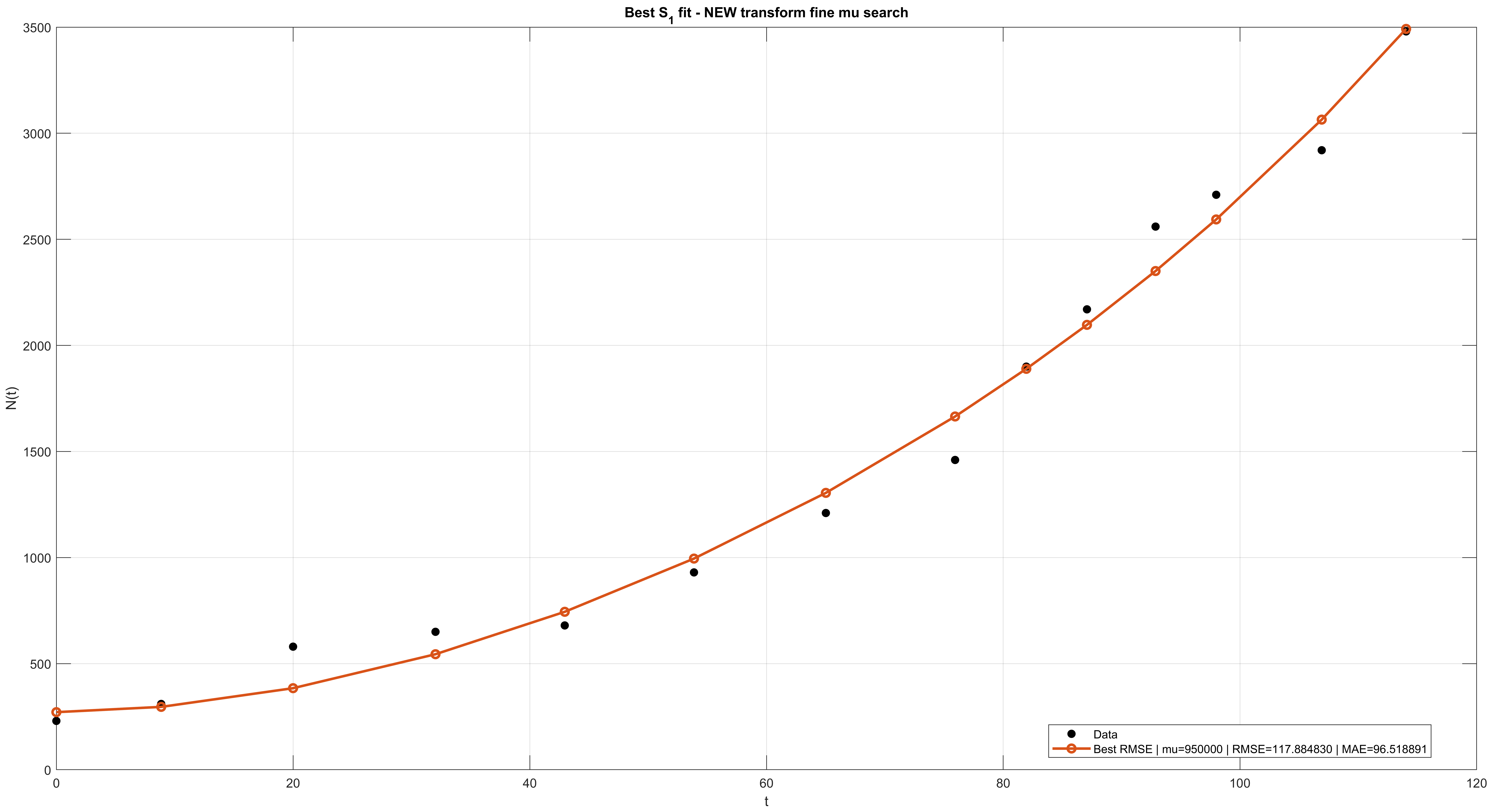}
    \caption{First-order approximation \(S_1(y_\mu(t))\) before the Chebyshev--Pad\'e reconstruction 
(RMSE = 117.88, MAE = 96.52, MAX = 209.43).}
    \label{fig:S1-raw-mu}
\end{figure}

After the transformation was fixed, Maple 2026 was used to construct a Chebyshev--Pad\'e approximation directly from the analytical expression
\[
S_1(y)=N_0(y)+N_1(y).
\]
The approximation was built in the transformed variable \(y=y_\mu(t)\), not in the original time variable \(t\). Maple's \texttt{chebpade} command was applied on an interval \([0,y_{\max}]\). The transformed data range was
\[
y_\mu(114.039)\approx 3.95690,
\]
and a slight refinement of the representation endpoint was tested. The best result was obtained with
\[
y_{\max}^{\mathrm{Cheb}}\approx 3.98440,
\qquad
[L/M]=[1/2].
\]
This adjustment changes only the interval where the Chebyshev--Pad\'e approximation is constructed; it does not refit the data and does not change the optimized transformation parameters.

The corresponding rational approximation obtained by Maple is reported below, with coefficients rounded for readability:
\[
\begin{aligned}
R_{1,2}^{(\mu)}(y)
&=
-\frac{-595.942421402264+41.707140894893\,y}
{2.346738697429-1.030322069114\,y+0.118567881527\,y^2}.
\end{aligned}
\]

The resulting approximation is evaluated as \(R_{1,2}^{(\mu)}(y_\mu(t))\), giving a rational semi-analytical representation in the original time variable. The numerical comparison in Figure~\ref{fig:real-chebpade-mu-ode45-rk4} and Table~\ref{Tmu} shows that this compact representation remains competitive with the numerical references and with the non-scaled generalized transform. The value of the framework lies in preserving an explicit low-order transform-domain structure while maintaining admissibility over the full data interval.

\begin{figure}[H]
    \centering
    \includegraphics[
        width=0.95\textwidth,
        height=8cm,
        keepaspectratio
    ]{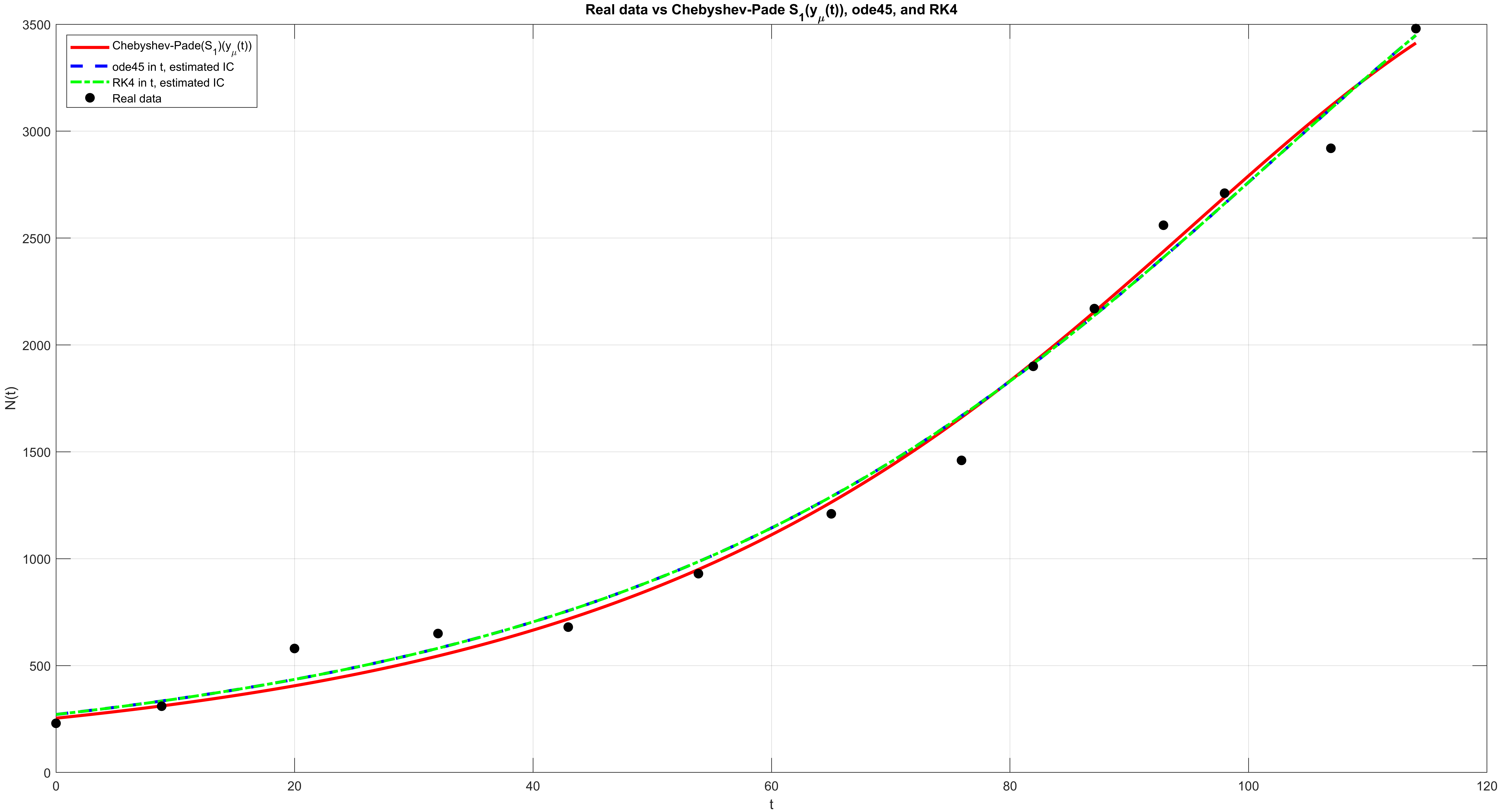}
    \caption{Real data versus the Chebyshev--Pad\'e approximation obtained from \(S_1(y_\mu(t))\), \texttt{ode45}, and RK4.}
    \label{fig:real-chebpade-mu-ode45-rk4}
\end{figure}

\begin{table}[H]
\centering
\caption{Error comparison for the Chebyshev--Pad\'e approximation obtained from \(S_1(y_\mu(t))\).}
\label{Tmu}
\begin{tabular}{|c|c|c|c|}
\hline
\textbf{Method} & \textbf{RMSE} & \textbf{MAE} & \textbf{MAX} \\
\hline
\(R_{1,2}^{CC_{T,\mu}^{\alpha}}(y_\mu(t))\)
& 102.47
& 75.51
& 201.22 \\
\hline
\(R_{1,2}^{\mathcal{L}_{T}^{\alpha}}(x(t))\)
& 102.51
& 75.03
& 202.26 \\
\hline
\(\texttt{ode45}\ \text{estIC}\)
& 102.94
& 82.91
& 208.38 \\
\hline
\(\text{RK4}\ \text{estIC}\)
& 102.94
& 82.91
& 208.38 \\
\hline
\end{tabular}
\end{table}

\section{Inverse Operator for the $\mu$-Scaled Generalized Transform}
\label{sec:inverse_mu}

\begin{Theorem}\label{t:inverse_mu}
Let \(f:[0,\infty)\to \RR\) be a function such that
\(CC_{T,\mu}^{\alpha}[f](s,\mu)\) exists for some \(s\), with \(0<\alpha\leq 1\) and \(0<\mu<\infty\). Then
\[
f(t)
=
\frac{1}{2\pi i\,\mu}
\lim_{R\to\infty}
\int_{a-iR}^{a+iR}
e^{s y_\mu(t)}
CC_{T,\mu}^{\alpha}[f](s,\mu)\,ds,
\]
where
\[
y_\mu(t)=\mu x\left(\frac{t}{\mu}\right),
\qquad
x(t)=\int_0^t \frac{d\omega}{T(\omega,\alpha)}.
\]
The integration is performed along the vertical line
\(\{\Re s=a\}\), whenever this line is contained in the region of
convergence of \(CC_{T,\mu}^{\alpha}[f](s,\mu)\).
\end{Theorem}

The proof is given in Appendix~\ref{app:proof_inverse_mu}.

\section{Residue Inversion Formulas for the Generalized Transforms}
\label{sec:residue_formulas}

\begin{Theorem}\label{t:residue_generalized_transform}
Let \(F(s)=\mathcal{L}_{T}^{\alpha}[f](s)\), and suppose that \(F\) admits a meromorphic continuation to the region enclosed by the Bromwich contour, with isolated poles \(s_1,\ldots,s_n\), and that the integral over the closing arc vanishes as the radius tends to infinity. Then
\[
f(t)
=
\frac{1}{2\pi i}
\lim_{R\to\infty}
\int_{a-iR}^{a+iR}
E_{\alpha}(s,t)F(s)\,ds
=
\sum_{k=1}^{n}
\operatorname{Res}
\left(
E_{\alpha}(s,t)F(s),s_k
\right),
\]
where the vertical line \(\Re(s)=a\) is chosen in the region of convergence of \(F\), and
\[
E_{\alpha}(s,t)
=
\exp\left(
s\int_0^t\frac{d\omega}{T(\omega,\alpha)}
\right).
\]
\end{Theorem}

The proof is given in Appendix~\ref{app:proof_residue_generalized_transform}.

\begin{Theorem}\label{t:residue_mu_transform}
Let \(F(s,\mu)=CC_{T,\mu}^{\alpha}[f](s,\mu)\), with \(0<\mu<\infty\), and suppose that \(F(\cdot,\mu)\) admits a meromorphic continuation to the region enclosed by the Bromwich contour, with isolated poles \(s_1,\ldots,s_n\), and that the integral over the closing arc vanishes as the radius tends to infinity. Then
\[
f(t)
=
\frac{1}{2\pi i\,\mu}
\lim_{R\to\infty}
\int_{a-iR}^{a+iR}
e^{s y_\mu(t)}F(s,\mu)\,ds
=
\frac{1}{\mu}
\sum_{k=1}^{n}
\operatorname{Res}
\left(
e^{s y_\mu(t)}F(s,\mu),s_k
\right),
\]
where the vertical line \(\Re(s)=a\) is chosen in the region of convergence of \(F(\cdot,\mu)\), and
\[
y_\mu(t)=\mu x\left(\frac{t}{\mu}\right),
\qquad
x(t)=\int_0^t\frac{d\omega}{T(\omega,\alpha)}.
\]
\end{Theorem}

The proof is given in Appendix~\ref{app:proof_residue_mu_transform}.

\section{Admissibility and Fixed-Point Consistency of the First-Order Approximation}
\label{sec:error_admissibility}

The purpose of this section is threefold. First, we define an admissible parameter region for the transformed tumor-growth model and state convergence hypotheses for the transform--Adomian components and Adomian polynomials. Second, we give conditions under which the resulting series solves the transformed equation and prove the existence of optimal parameters on compact admissible sets. Third, we justify the first truncation $S_1$ through a fixed-point formulation and derive the corresponding interior optimality conditions for the transform-parameter fitting problem.

Convergence of Adomian-type decompositions has been studied under suitable
assumptions for nonlinear differential and integral equations \cite{cherruault1993,abbaoui1994,abdelrazec2011}.
The operational properties of the generalized Laplace transform used in this work,
including its relation with the classical Laplace transform, are given in \cite{bosch2021}.

We use a unified notation for the two transformed variables used in this paper. For
the generalized transform, set
\[
r_z(t)=x_z(t),
\]
where
\[
x_z(t)
=
 \sigma^{\alpha-1}
\frac{(1+c(1-\alpha)t)^{1-p}-1}{c(1-\alpha)(1-p)}.
\]
For the \(\mu\)-scaled generalized transform, \(\mu>0\) is fixed in advance and we set
\[
r_z(t)=y_{\mu,z}(t)=\mu x_z(t/\mu),
\]
that is,
\[
y_{\mu,z}(t)
=
\mu  \sigma^{\alpha-1}
\frac{\left(1+c(1-\alpha)\frac{t}{\mu}\right)^{1-p}-1}
{c(1-\alpha)(1-p)}.
\]
Throughout this section, \(r_z(t)\) denotes either \(x_z(t)\) or \(y_{\mu,z}(t)\), according
to the transform under consideration. In both cases, $z=(\alpha,c, \sigma,p)$.

We restrict the analysis to the parametric regime used in the optimization:
\[
0<\alpha<1,\qquad  \sigma>0,\qquad c\neq 0,\qquad p\neq 1,
\]
and, in the \(\mu\)-scaled case, \(\mu>0\) is assumed fixed. We also assume
\[
k_1>0,
\]
as in the logistic--Allee model considered in this work. Let
\[
h_1=\eta-\frac{k_4}{k_1},
\qquad
h_2=\frac{k_4}{k_1},
\qquad
h_3=\frac{k_4}{k_1}+d.
\]
In what follows, \(N_0\) and \(N_1\) are regarded as functions of the transformed variable
\[
\rho=r_z(t).
\]
Thus \(N_j(r_z(t))\) denotes the composition of the \(j\)-th Adomian component with the
corresponding admissible transformed variable.

\begin{Theorem}
\label{thm:admissibility_finiteness_S1}
Let \(I=[0,t_f]\), with \(t_f<\infty\). Assume that the transformed variable \(r_z\) is
well defined on \(I\). More precisely, assume that
\[
1+c(1-\alpha)t>0,\qquad t\in I,
\]
in the generalized-transform case, or
\[
1+c(1-\alpha)\frac{t}{\mu}>0,\qquad t\in I,
\]
in the \(\mu\)-scaled case. Assume also that
\[
h_1+h_3e^{-k_1r_z(t)}>0,\qquad t\in I.
\]
Then
\[
N_0(r_z(t))=h_1e^{k_1r_z(t)}+h_2
\]
and \(N_1(r_z(t))\) are real-valued, continuous, and bounded on \(I\). Consequently,
\[
S_1(r_z(t))=N_0(r_z(t))+N_1(r_z(t))
\]
is real-valued, continuous, and bounded on \(I\).

Moreover, if the data set
\[
\{(t_i,N_i^{\mathrm{data}})\}_{i=1}^{m_d}
\]
is finite, with \(t_i\in I\) and \(N_i^{\mathrm{data}}\) finite, then
\[
J(z)
=
\left[
\frac{1}{m_d}
\sum_{i=1}^{m_d}
\left|S_1(r_z(t_i))-N_i^{\mathrm{data}}\right|^2
\right]^{1/2}
<\infty.
\]
\end{Theorem}

\begin{proof}
In the generalized-transform case, the condition
\[
1+c(1-\alpha)t>0,\qquad t\in I,
\]
implies that \(T(t,\alpha;z)\) is real-valued, positive, and finite on \(I\). Hence
\(1/T(t,\alpha;z)\) is continuous on \(I\), and \(x_z(t)\) is finite and continuous on \(I\).

In the \(\mu\)-scaled case, the condition
\[
1+c(1-\alpha)\frac{t}{\mu}>0,\qquad t\in I,
\]
with fixed \(\mu>0\), implies that \(T(t/\mu,\alpha;z)\) is real-valued, positive, and
finite on \(I\). Hence \(y_{\mu,z}(t)=\mu x_z(t/\mu)\) is finite and continuous on \(I\).

Thus, in both cases, \(r_z(t)\) is finite and continuous on \(I\). Therefore
\(e^{k_1r_z(t)}\) is finite and continuous on \(I\), and
\[
N_0(r_z(t))=h_1e^{k_1r_z(t)}+h_2
\]
is finite and continuous on \(I\). Moreover,
\[
N_0(r_z(t))+d=h_1e^{k_1r_z(t)}+h_3.
\]
Since
\[
h_1+h_3e^{-k_1r_z(t)}>0
\]
and \(e^{k_1r_z(t)}>0\), it follows that
\[
h_1e^{k_1r_z(t)}+h_3>0.
\]
Consequently,
\[
N_0(r_z(t))+d>0,\qquad t\in I.
\]
Since \(N_0(r_z(t))+d\) is continuous and positive on \(I\), there exists
\[
m_z=\min_{t\in I}\bigl(N_0(r_z(t))+d\bigr)>0.
\]
Thus
\[
\frac{1}{N_0(r_z(t))+d}
\]
is real-valued, continuous, and bounded on \(I\).

The condition
\[
h_1+h_3e^{-k_1r_z(t)}>0
\]
also guarantees that the logarithmic contribution appearing in \(N_1(r_z(t))\) is
real-valued, finite, and continuous on \(I\). Moreover,
\[
h_1+h_3=\eta+d>0,
\]
so the denominator appearing in the logarithmic term is positive. Therefore the above
admissibility condition is sufficient to ensure that the logarithmic term is real-valued.

The remaining terms in \(N_1(r_z(t))\) are finite sums and products of continuous
functions involving \(r_z(t)\), exponentials, and polynomial factors. Therefore
\[
N_1(r_z(\cdot))\in C(I).
\]
Since \(I\) is compact,
\[
\|N_0(r_z(\cdot))\|_{\infty,I}<\infty,
\qquad
\|N_1(r_z(\cdot))\|_{\infty,I}<\infty.
\]
Therefore,
\[
S_1(r_z(\cdot))\in C(I),
\qquad
\|S_1(r_z(\cdot))\|_{\infty,I}<\infty.
\]
Since the data set is finite and each \(N_i^{\mathrm{data}}\) is finite, one has
\[
\left|S_1(r_z(t_i))-N_i^{\mathrm{data}}\right|<\infty,
\qquad i=1,\ldots,m_d.
\]
Hence \(J(z)<\infty\).
\end{proof}

\begin{Theorem}
\label{thm:admissible_minimizer}
Let \(K_{\mathrm{adm}}\) be a nonempty compact subset strictly contained in the admissible
parameter region, so that the admissibility inequalities in
Theorem~\ref{thm:admissibility_finiteness_S1} hold on \(I\) for every
\(z\in K_{\mathrm{adm}}\). In the \(\mu\)-scaled case, \(\mu>0\) is fixed. Assume that, for
each \(i=1,\ldots,m_d\), the map
\[
z\mapsto S_1(r_z(t_i))
\]
is continuous on \(K_{\mathrm{adm}}\). Then there exists \(z^*\in K_{\mathrm{adm}}\) such that
\[
J(z^*)=\min_{z\in K_{\mathrm{adm}}}J(z).
\]
Moreover,
\[
S_1(r_{z^*}(\cdot))\in C(I),
\qquad
\|S_1(r_{z^*}(\cdot))\|_{\infty,I}<\infty,
\qquad
J(z^*)<\infty.
\]
\end{Theorem}

\begin{proof}
By Theorem~\ref{thm:admissibility_finiteness_S1},
\[
J(z)<\infty,\qquad z\in K_{\mathrm{adm}}.
\]
For each \(i=1,\ldots,m_d\), the function
\[
z\mapsto
\left|S_1(r_z(t_i))-N_i^{\mathrm{data}}\right|^2
\]
is continuous on \(K_{\mathrm{adm}}\). Hence
\[
J^2(z)
=
\frac{1}{m_d}
\sum_{i=1}^{m_d}
\left|S_1(r_z(t_i))-N_i^{\mathrm{data}}\right|^2
\]
is continuous on \(K_{\mathrm{adm}}\). Since \(J(z)=\sqrt{J^2(z)}\), it follows that
\[
J\in C(K_{\mathrm{adm}}).
\]
Since \(K_{\rm adm}\) is nonempty and compact, and since \(J\in C(K_{\rm adm})\), the existence of a minimizer follows directly from the Weierstrass Extreme Value Theorem. Hence there exists an optimal parameter vector
\[
z^*\in K_{\rm adm}
\]
such that
\[
J(z^*)=\min_{z\in K_{\rm adm}}J(z).
\]
The remaining assertions follow from Theorem~\ref{thm:admissibility_finiteness_S1}
with \(z=z^*\).
\end{proof}

For an admissible transformed variable \(\rho=r_z(t)\), the transformed form of the
logistic--Allee equation is
\[
\frac{dN}{d\rho}
=
k_1N+\frac{k_2}{N+d}-k_3N^2-k_4,
\qquad
N(0)=\eta.
\]
The function
\[
N_0(\rho)=h_1e^{k_1\rho}+h_2
\]
satisfies
\[
N_0'(\rho)=k_1N_0(\rho)-k_4,
\qquad
N_0(0)=\eta.
\]
Therefore, writing
\[
N(\rho)=N_0(\rho)+U(\rho),
\]
the correction \(U\) satisfies the transformed Volterra equation
\[
U(\rho)
=
\int_0^\rho
e^{k_1(\rho-\xi)}
\left[
\frac{k_2}{N_0(\xi)+U(\xi)+d}
-
k_3\bigl(N_0(\xi)+U(\xi)\bigr)^2
\right]\,d\xi.
\]

\begin{Theorem}
\label{thm:fixed_point_convergence_transform_adomian}
Let \(r_z(t)\) denote either \(x_z(t)\) or \(y_{\mu,z}(t)\). Under the admissibility
assumptions and the standing parameter restrictions, \(T>0\) on the relevant interval.
Hence the transformed variable \(r_z\) is strictly increasing on \(I\), with \(r_z(0)=0\).
Define
\[
R_z=[0,r_z(t_f)].
\]
Let
\[
N_0(\rho)=h_1e^{k_1\rho}+h_2,
\qquad \rho\in R_z.
\]
Let \(C(R_z)\) be endowed with the norm
\[
\|U\|_{\infty,R_z}=\max_{\rho\in R_z}|U(\rho)|.
\]
For \(R>0\), define
\[
B_R=\{U\in C(R_z):\|U\|_{\infty,R_z}\le R\}.
\]
Assume that
\[
m_z(R)=\min_{\rho\in R_z}\bigl(N_0(\rho)+d\bigr)-R>0
\]
and set
\[
M_z(R)=\max_{\rho\in R_z}|N_0(\rho)|+R.
\]
Define
\[
\mathcal A_z(U)(\rho)
=
\int_0^\rho
e^{k_1(\rho-\xi)}
\left[
\frac{k_2}{N_0(\xi)+U(\xi)+d}
-
k_3\bigl(N_0(\xi)+U(\xi)\bigr)^2
\right]\,d\xi.
\]
Set
\[
\Gamma_z
=
\sup_{0\le \rho\le r_z(t_f)}
\int_0^\rho e^{k_1(\rho-\xi)}\,d\xi.
\]
Since \(k_1>0\), equivalently,
\[
\Gamma_z=\frac{e^{k_1r_z(t_f)}-1}{k_1}.
\]
Assume that
\[
\Gamma_z
\left(
\frac{|k_2|}{m_z(R)}
+
|k_3|M_z(R)^2
\right)
\le R,
\]
and
\[
\delta_z(R)
=
\Gamma_z
\left(
\frac{|k_2|}{m_z(R)^2}
+
2|k_3|M_z(R)
\right)
<1.
\]
Here \(\delta_z(R)\) is the contraction constant associated with \(\mathcal A_z\).

Then \(\mathcal A_z\) has a unique fixed point \(U^*\in B_R\). Consequently,
\[
N^*(\rho)=N_0(\rho)+U^*(\rho)
\]
satisfies the transformed Volterra integral equation associated with the logistic--Allee
model on \(R_z\). Moreover, \(N^*\) satisfies the transformed differential equation
\[
(N^*)'(\rho)
=
k_1N^*(\rho)
+
\frac{k_2}{N^*(\rho)+d}
-
k_3(N^*(\rho))^2
-
k_4,
\qquad
N^*(0)=\eta.
\]

Let \(\{\widetilde U_n\}_{n\ge 0}\) be the Picard sequence defined by
\[
\widetilde U_0=0,
\qquad
\widetilde U_{n+1}=\mathcal A_z(\widetilde U_n),
\qquad n\ge 0.
\]
Then
\[
\lim_{n\to\infty}
\|\widetilde U_n-U^*\|_{\infty,R_z}=0.
\]
Consequently,
\[
\lim_{n\to\infty}
\|N_0+\widetilde U_n-(N_0+U^*)\|_{\infty,R_z}=0.
\]
Equivalently,
\[
\lim_{n\to\infty}
\left\|
N_0(r_z(\cdot))+\widetilde U_n(r_z(\cdot))
-
\bigl[N_0(r_z(\cdot))+U^*(r_z(\cdot))\bigr]
\right\|_{\infty,I}
=0.
\]

If, in addition, the Adomian correction series
\[
U_A(\rho)=\sum_{n=1}^{\infty}N_n(\rho)
\]
converges uniformly on \(R_z\), its partial sums
\[
U_n=N_1+\cdots+N_n
\]
remain in \(B_R\), and the nonlinear series
\[
\frac{1}{N_0+U_A+d}
\qquad \text{and} \qquad
(N_0+U_A)^2
\]
are the uniform limits on \(R_z\) of the corresponding Adomian expansions, then
\[
U_A=U^*.
\]
Hence the Adomian series represents the unique admissible fixed point of the transformed
correction equation. In particular,
\[
S_1=N_0+N_1
\]
is the first nontrivial Adomian truncation associated with the admissible fixed-point
problem.
\end{Theorem}

\begin{proof}
\emph{Step A: Invariance of the admissible ball.}

Let \(U\in B_R\). Then
\[
\|U\|_{\infty,R_z}\le R.
\]
For every \(\xi\in R_z\), we have
\[
N_0(\xi)+U(\xi)+d
\ge
N_0(\xi)+d-|U(\xi)|
\ge
\min_{\rho\in R_z}\bigl(N_0(\rho)+d\bigr)-R
=
m_z(R)>0.
\]
Therefore,
\[
\left|
\frac{1}{N_0(\xi)+U(\xi)+d}
\right|
\le
\frac{1}{m_z(R)}.
\]
Moreover,
\[
|N_0(\xi)+U(\xi)|
\le
\max_{\rho\in R_z}|N_0(\rho)|+\|U\|_{\infty,R_z}
\le
M_z(R).
\]
Since \(U\in C(R_z)\) and \(N_0+U+d\ge m_z(R)>0\), the integrand defining
\(\mathcal A_z(U)\) is continuous on \(R_z\). Hence
\[
\mathcal A_z(U)\in C(R_z).
\]

For every \(\rho\in R_z\), we obtain
\[
|\mathcal A_z(U)(\rho)|
\le
\int_0^\rho
e^{k_1(\rho-\xi)}
\left[
\frac{|k_2|}{m_z(R)}
+
|k_3|M_z(R)^2
\right]\,d\xi.
\]
By the definition of \(\Gamma_z\),
\[
|\mathcal A_z(U)(\rho)|
\le
\Gamma_z
\left[
\frac{|k_2|}{m_z(R)}
+
|k_3|M_z(R)^2
\right].
\]
By the first assumed inequality,
\[
\Gamma_z
\left(
\frac{|k_2|}{m_z(R)}
+
|k_3|M_z(R)^2
\right)
\le R.
\]
Thus,
\[
\|\mathcal A_z(U)\|_{\infty,R_z}\le R,
\]
and consequently
\[
\mathcal A_z(B_R)\subset B_R.
\]

\emph{Step B: Contraction estimate.}

We now prove that \(\mathcal A_z\) is a contraction on \(B_R\). Let \(U,V\in B_R\).
Then, for every \(\xi\in R_z\),
\[
|N_0(\xi)+U(\xi)+d|\ge m_z(R),
\qquad
|N_0(\xi)+V(\xi)+d|\ge m_z(R).
\]
Hence
\[
\left|
\frac{1}{N_0(\xi)+U(\xi)+d}
-
\frac{1}{N_0(\xi)+V(\xi)+d}
\right|
\le
\frac{1}{m_z(R)^2}|U(\xi)-V(\xi)|.
\]
Also,
\[
|N_0(\xi)+U(\xi)|\le M_z(R),
\qquad
|N_0(\xi)+V(\xi)|\le M_z(R).
\]
Therefore,
\[
\left|
\bigl(N_0(\xi)+U(\xi)\bigr)^2
-
\bigl(N_0(\xi)+V(\xi)\bigr)^2
\right|
\le
2M_z(R)|U(\xi)-V(\xi)|.
\]
It follows that
\[
\begin{aligned}
|\mathcal A_z(U)(\rho)-\mathcal A_z(V)(\rho)|
&\le
\int_0^\rho
e^{k_1(\rho-\xi)}
\left(
\frac{|k_2|}{m_z(R)^2}
+
2|k_3|M_z(R)
\right)
|U(\xi)-V(\xi)|\,d\xi  \\
&\le
\Gamma_z
\left(
\frac{|k_2|}{m_z(R)^2}
+
2|k_3|M_z(R)
\right)
\|U-V\|_{\infty,R_z}.
\end{aligned}
\]
By definition of \(\delta_z(R)\), this gives
\[
|\mathcal A_z(U)(\rho)-\mathcal A_z(V)(\rho)|
\le
\delta_z(R)\|U-V\|_{\infty,R_z}.
\]
Taking the supremum over \(\rho\in R_z\), we obtain
\[
\|\mathcal A_z(U)-\mathcal A_z(V)\|_{\infty,R_z}
\le
\delta_z(R)\|U-V\|_{\infty,R_z}.
\]

\emph{Step C: Application of Banach's fixed-point theorem.}

Since \(\delta_z(R)<1\), the operator \(\mathcal A_z\) is a contraction on \(B_R\).
Since \(B_R\) is a closed subset of the Banach space \(C(R_z)\), it is complete.
Therefore, by Banach's fixed-point theorem, there exists a unique \(U^*\in B_R\) such
that
\[
\mathcal A_z(U^*)=U^*.
\]
Consequently,
\[
N^*(\rho)=N_0(\rho)+U^*(\rho)
\]
satisfies the transformed Volterra integral equation associated with the logistic--Allee
model on \(R_z\).

\emph{Step D: Recovery of the transformed differential equation.}

Since the integrand in the Volterra equation is continuous, \(U^*\) is continuously
differentiable on \(R_z\). Differentiating the Volterra equation gives
\[
(U^*)'(\rho)
=
k_1U^*(\rho)
+
\frac{k_2}{N_0(\rho)+U^*(\rho)+d}
-
k_3\bigl(N_0(\rho)+U^*(\rho)\bigr)^2.
\]
Using
\[
N^*(\rho)=N_0(\rho)+U^*(\rho)
\]
and
\[
N_0'(\rho)=k_1N_0(\rho)-k_4,
\]
we obtain
\[
(N^*)'(\rho)
=
k_1N^*(\rho)
+
\frac{k_2}{N^*(\rho)+d}
-
k_3(N^*(\rho))^2
-
k_4.
\]
Moreover,
\[
N^*(0)=N_0(0)+U^*(0)=\eta,
\]
because the Volterra equation gives \(U^*(0)=0\). Hence \(N^*\) satisfies the transformed
differential equation.

\emph{Step E: Picard convergence.}

The convergence of the Picard sequence follows directly from Banach's fixed-point
theorem. Hence
\[
\lim_{n\to\infty}
\|\widetilde U_n-U^*\|_{\infty,R_z}=0.
\]
Adding \(N_0\) to both terms gives
\[
\lim_{n\to\infty}
\|N_0+\widetilde U_n-(N_0+U^*)\|_{\infty,R_z}=0.
\]
Since \(r_z(I)=R_z\), composition with \(r_z(t)\) yields
\[
\lim_{n\to\infty}
\left\|
N_0(r_z(\cdot))+\widetilde U_n(r_z(\cdot))
-
\bigl[N_0(r_z(\cdot))+U^*(r_z(\cdot))\bigr]
\right\|_{\infty,I}
=0.
\]

\emph{Step F: Connection with the Adomian expansion.}

It remains to explain the connection with the Adomian expansion. On \(B_R\), the
denominator
\[
N_0+U+d
\]
is bounded away from zero by \(m_z(R)>0\). Hence the map
\[
U\mapsto \frac{1}{N_0+U+d}
\]
is analytic on the admissible ball \(B_R\). Therefore, the associated Adomian polynomials
for the rational term are well defined. The map
\[
U\mapsto (N_0+U)^2
\]
is polynomial, so its Adomian expansion is given by the corresponding Cauchy product.

Assume now that the Adomian correction series
\[
U_A(\rho)=\sum_{n=1}^{\infty}N_n(\rho)
\]
converges uniformly on \(R_z\), that its partial sums
\[
U_n=N_1+\cdots+N_n
\]
remain in \(B_R\), and that the nonlinear series
\[
\frac{1}{N_0+U_A+d}
\qquad \text{and} \qquad
(N_0+U_A)^2
\]
are the uniform limits on \(R_z\) of the corresponding Adomian expansions. Then \(U_A\)
may be substituted into the transformed correction equation, and therefore
\[
\mathcal A_z(U_A)=U_A.
\]
By uniqueness of the fixed point in \(B_R\), it follows that
\[
U_A=U^*.
\]
Thus, under the stated additional admissibility and uniform-convergence hypotheses, the
Adomian series represents the unique admissible fixed point of the transformed correction
equation. In particular,
\[
S_1=N_0+N_1
\]
is the first nontrivial Adomian truncation associated with the admissible fixed-point
problem.
\end{proof}

\begin{Remark}
Theorem~\ref{thm:fixed_point_convergence_transform_adomian} establishes the existence
and uniqueness of the admissible solution of the transformed correction equation by
Banach's fixed-point theorem. Under the additional uniform convergence assumptions
stated in the theorem, this fixed point coincides with the limit of the corresponding
Adomian correction series. Hence the first-order approximation $S_1$
is the first nontrivial Adomian truncation of the admissible fixed-point solution.

The following result quantifies the error introduced by truncating the admissible
correction expansion after finitely many terms and provides a practical stopping
criterion for computations.
\end{Remark}

\begin{Corollary}
\label{cor:first_order_truncation_estimate}
Under the hypotheses of Theorem~\ref{thm:fixed_point_convergence_transform_adomian},
let
\[
\widetilde U_0=0,
\qquad
\widetilde U_{n+1}=\mathcal A_z(\widetilde U_n).
\]
Then
\[
\|\widetilde U_n-U^*\|_{\infty,R_z}
\le
\frac{\delta_z(R)}{1-\delta_z(R)}
\|\widetilde U_n-\widetilde U_{n-1}\|_{\infty,R_z},
\qquad n\ge 1.
\]
In particular, since the first Picard correction coincides with the first Adomian correction,
\[
\widetilde U_1=\mathcal A_z(0)=N_1,
\]
one obtains
\[
\|U^*-N_1\|_{\infty,R_z}
\le
\frac{\delta_z(R)}{1-\delta_z(R)}
\|N_1\|_{\infty,R_z}.
\]
Consequently,
\[
\|N_0+U^*-(N_0+N_1)\|_{\infty,R_z}
\le
\frac{\delta_z(R)}{1-\delta_z(R)}
\|N_1\|_{\infty,R_z}.
\]
Equivalently,
\[
\|N_0(r_z(\cdot))+U^*(r_z(\cdot))-S_1(r_z(\cdot))\|_{\infty,I}
\le
\frac{\delta_z(R)}{1-\delta_z(R)}
\|N_1\|_{\infty,R_z}.
\]
\end{Corollary}

\begin{proof}
Since \(\mathcal A_z\) is a contraction with constant \(\delta_z(R)<1\), the standard
a posteriori estimate for fixed-point iterations gives
\[
\|\widetilde U_n-U^*\|_{\infty,R_z}
\le
\frac{\delta_z(R)}{1-\delta_z(R)}
\|\widetilde U_n-\widetilde U_{n-1}\|_{\infty,R_z},
\qquad n\ge 1.
\]
Taking \(n=1\), and using
\[
\widetilde U_0=0,
\qquad
\widetilde U_1=\mathcal A_z(0)=N_1,
\]
one obtains
\[
\|U^*-N_1\|_{\infty,R_z}
\le
\frac{\delta_z(R)}{1-\delta_z(R)}
\|N_1\|_{\infty,R_z}.
\]
Adding \(N_0\) to both terms does not change the norm of the difference. Composition
with \(r_z(t)\) gives the estimate on \(I\).
\end{proof}

\subsection{Admissible Regions and Optimality for the Parametric Kernel}
\label{subsec:parametric_kernel_admissibility}

The following result specializes the admissibility conditions to the parametric kernel used in the numerical construction.

\begin{Proposition}
\label{prop:admissible_region_parametric_kernel}
Let \(I=[0,t_f]\), \(t_f<\infty\), and let
\[
T(t,\alpha;z)= \sigma^{1-\alpha}\bigl(1+c(1-\alpha)t\bigr)^p,
\qquad 
z=(\alpha,c, \sigma,p),
\]
where
\[
0<\alpha<1,\qquad  \sigma>0,\qquad c\neq 0,\qquad p\neq 1.
\]
For \(\mu>0\), define
\[
y_{\mu,z}(t)
=
\mu  \sigma^{\alpha-1}
\frac{
\left(1+c(1-\alpha)\dfrac{t}{\mu}\right)^{1-p}-1
}{
c(1-\alpha)(1-p)
}.
\]
Assume
\[
k_1>0,\qquad h_3>0,\qquad h_1<0,\qquad h_1+h_3>0.
\]
Set
\[
y_{\mathrm{crit}}
=
-\frac{1}{k_1}\ln\left(-\frac{h_1}{h_3}\right).
\]
Then \(y_{\mathrm{crit}}>0\). Moreover, \((z,\mu)\) is admissible on \(I\) if and only if
\[
1+c(1-\alpha)\frac{t_f}{\mu}>0
\]
and
\[
y_{\mu,z}(t_f)<y_{\mathrm{crit}}.
\]
Equivalently, if
\[
\Omega_{\alpha,c,p,\mu}(t_f)
=
\mu
\frac{
\left(1+c(1-\alpha)\dfrac{t_f}{\mu}\right)^{1-p}-1
}{
c(1-\alpha)(1-p)
},
\]
then
\[
y_{\mu,z}(t_f)= \sigma^{\alpha-1}\Omega_{\alpha,c,p,\mu}(t_f),
\]
and, for fixed \((\alpha,c,p,\mu)\) satisfying
\[
1+c(1-\alpha)\frac{t_f}{\mu}>0,
\]
the admissible values of $\sigma$ are precisely
\[
 \sigma>
\left(
\frac{\Omega_{\alpha,c,p,\mu}(t_f)}
{y_{\mathrm{crit}}}
\right)^{\frac{1}{1-\alpha}}.
\]
The non-scaled case is obtained by taking \(\mu=1\).
\end{Proposition}

\begin{proof}
The condition \(T(t/\mu,\alpha;z)>0\) on \(I\) is equivalent to
\[
1+c(1-\alpha)\frac{t}{\mu}>0,
\qquad 0\leq t\leq t_f.
\]
Since this expression is affine in \(t\) and equals \(1\) at \(t=0\), this is equivalent to
\[
1+c(1-\alpha)\frac{t_f}{\mu}>0.
\]
Under this condition,
\[
y'_{\mu,z}(t)=\frac{1}{T(t/\mu,\alpha;z)}>0,
\]
so \(y_{\mu,z}\) is strictly increasing on \(I\).

The logarithmic admissibility condition is
\[
h_1+h_3e^{-k_1y_{\mu,z}(t)}>0.
\]
Since \(k_1>0\), \(h_3>0\), \(h_1<0\), and \(h_1+h_3>0\), the function
\[
y\mapsto h_1+h_3e^{-k_1y}
\]
has the unique positive zero
\[
y_{\mathrm{crit}}
=
-\frac{1}{k_1}\ln\left(-\frac{h_1}{h_3}\right).
\]
Hence the logarithmic condition holds on \(I\) if and only if
\[
y_{\mu,z}(t)<y_{\mathrm{crit}},
\qquad 0\leq t\leq t_f.
\]
Since \(y_{\mu,z}\) is strictly increasing, this is equivalent to
\[
y_{\mu,z}(t_f)<y_{\mathrm{crit}}.
\]

Finally,
\[
y_{\mu,z}(t_f)= \sigma^{\alpha-1}\Omega_{\alpha,c,p,\mu}(t_f).
\]
Because \(0<\alpha<1\),
\[
 \sigma^{\alpha-1}=\frac{1}{\sigma^{1-\alpha}}.
\]
Observe that
\[
\Omega_{\alpha,c,p,\mu}(t_f)>0,
\]
since it is obtained from an integral of positive functions over a nontrivial interval. Consequently, the inequality
\[
\sigma^{\alpha-1}\Omega_{\alpha,c,p,\mu}(t_f)<y_{\rm crit}
\]
is equivalent to
\[
\sigma>
\left(
\frac{\Omega_{\alpha,c,p,\mu}(t_f)}
{y_{\rm crit}}
\right)^{1/(1-\alpha)}.
\]

\end{proof}

\begin{Theorem}
\label{thm:interior_optimality_transform_region}
Let \(\mu>0\) be fixed, and let
\[
\mathcal D_\mu
=
\{z=(\alpha,c, \sigma,p):(z,\mu)\ \text{belongs to the admissible region described in Proposition~\ref{prop:admissible_region_parametric_kernel}}\}.
\]
For \(z=(\alpha,c, \sigma,p)\in \mathcal D_\mu\), define
\[
Q_\mu(z)
=
\frac{1}{m}
\sum_{i=1}^m
\left(
S_1(y_{\mu,z}(t_i))-N_i^{\mathrm{data}}
\right)^2.
\]
Assume that \(S_1\in C^1\) on the transformed data interval and that, for each
\(i=1,\ldots,m\), the map
\[
z\mapsto y_{\mu,z}(t_i)
\]
is \(C^1\) on \(\mathcal D_\mu\). If
\[
z^*=(\alpha^*,c^*, \sigma^*,p^*)
\]

is an interior local minimizer of \(Q_\mu\) in \(\mathcal D_\mu\), then the first-order
optimality conditions
\[
\frac{\partial Q_\mu}{\partial\vartheta}(z^*)=0,
\qquad
\vartheta\in\{\alpha,c,\sigma,p\},
\]
hold. Equivalently, for every
\[
\vartheta\in\{\alpha,c,\sigma,p\},
\]
one has

\[
\sum_{i=1}^m
\left(
S_1(y_{\mu,z^*}(t_i))-N_i^{\mathrm{data}}
\right)
S_1'(y_{\mu,z^*}(t_i))
\frac{\partial y_{\mu,z}}{\partial \vartheta}(t_i)
\bigg|_{z=z^*}
=0.
\]
\end{Theorem}

\begin{proof}
Let
\[
e_i(z)
=
S_1(y_{\mu,z}(t_i))-N_i^{\mathrm{data}},
\qquad i=1,\ldots,m.
\]
Then
\[
Q_\mu(z)=\frac{1}{m}\sum_{i=1}^m e_i(z)^2.
\]
For any component \(\vartheta\in\{\alpha,c, \sigma,p\}\), we have
\[
\frac{\partial Q_\mu}{\partial\vartheta}
=
\frac{2}{m}
\sum_{i=1}^m
e_i(z)
\frac{\partial e_i}{\partial\vartheta}.
\]
By the chain rule,
\[
\frac{\partial e_i}{\partial\vartheta}
=
S_1'(y_{\mu,z}(t_i))
\frac{\partial y_{\mu,z}}{\partial\vartheta}(t_i).
\]
Since \(z^*\) is an interior local minimizer of \(Q_\mu\) in \(\mathcal D_\mu\),
\[
\frac{\partial Q_\mu}{\partial\vartheta}(z^*)=0,
\qquad
\vartheta\in\{\alpha,c, \sigma,p\}.
\]
Therefore,
\[
\sum_{i=1}^m
\left(
S_1(y_{\mu,z^*}(t_i))-N_i^{\mathrm{data}}
\right)
S_1'(y_{\mu,z^*}(t_i))
\frac{\partial y_{\mu,z}}{\partial \vartheta}(t_i)
\bigg|_{z=z^*}
=0,
\qquad
\vartheta\in\{\alpha,c, \sigma,p\}.
\]
The result follows.
\end{proof}

\begin{Remark}
Since \(J_\mu(z)=\sqrt{Q_\mu(z)}\), the minimizers of \(J_\mu\) and \(Q_\mu\) coincide whenever \(Q_\mu(z)>0\). If a minimizer lies on the boundary of \(D_\mu\), the interior condition in Theorem~\ref{thm:interior_optimality_transform_region} is replaced by the corresponding constrained optimality conditions.
\end{Remark}

\section{Conclusions}
\label{sec:conclusions}

In this work, we developed a generalized transform framework for constructing admissible semi-analytical approximations to a logistic--Allee tumor-growth model. The framework combines a generalized Laplace transform, Adomian decomposition, Chebyshev--Pad\'e rational reconstruction, and a $\mu$-scaled generalized transform. The numerical results show that compact low-order representations can approximate experimental tumor-growth data with errors comparable to those obtained from standard numerical reference solutions.

Beyond the numerical implementation, the paper provides the analytical basis for the method through inverse relations, convolution identities, and residue inversion formulas. It also establishes admissibility conditions, proves the existence of an admissible minimizer, characterizes admissible regions for the parametric kernel family, derives an interior optimality condition for the transform-parameter fitting problem, and proves fixed-point consistency of the transformed correction equation together with a first-order truncation estimate.

The purpose of the paper is not to claim biological superiority over existing tumor-growth models, but rather to demonstrate how the generalized-transform framework can generate admissible semi-analytical representations for nonlinear differential equations. Future work will address higher-order corrections, broader nonlinear biological models, parameter sensitivity, identifiability, and error propagation within the generalized-transform setting.


\appendix
\section{Technical Proofs for the $\mu$-scaled Generalized Transform}
\label{app:mu_scaled_transform_proofs}

\subsection{Proof of Theorem~\ref{t:comp2}}
\label{app:proof_tcomp2}
\begin{proof}
Let us start by assuming that $D^{\lceil \a \rceil}f$ exists at the point $t$. Now take $q=h \, T\!\left(\frac{t}{\mu},\alpha\right)$ in the definition of $G_{T,\mu}^{\alpha}f(t)$,
\begin{equation*}
    \begin{aligned}
    G_{T,\mu}^{\alpha}f(t)
    &=
    T\!\left(\frac{t}{\mu},\alpha\right)^{{\lceil \alpha \rceil}}
    \lim_{q\to 0}
    \frac{1}{q^{{\lceil \alpha \rceil}}}
    \sum_{k=0}^{{\lceil \alpha \rceil}}
    (-1)^k
    \binom{{\lceil \alpha \rceil}}{k}
    f(t-kq)  \\
    &=
    T\!\left(\frac{t}{\mu},\alpha\right)^{{\lceil \alpha \rceil}}
    D^{\lceil \alpha \rceil}f(t).
    \end{aligned}
\end{equation*}
This proves $(1)$.

For $(2)$, let $0<\alpha\leq 1$. Then $\lceil\alpha\rceil=1$, and the definition gives
\[
G_{T,\mu}^{\alpha}f(t)
=
\lim_{h\to 0}
\frac{f(t)-f\!\left(t-hT\!\left(\frac{t}{\mu},\alpha\right)\right)}{h}.
\]
Taking $q=hT\!\left(\frac{t}{\mu},\alpha\right)$, we obtain
\[
G_{T,\mu}^{\alpha}f(t)
=
T\!\left(\frac{t}{\mu},\alpha\right)
\lim_{q\to 0}
\frac{f(t)-f(t-q)}{q}.
\]
Hence the above limit exists if and only if $f$ is differentiable at $t$, and in that case
\[
G_{T,\mu}^{\alpha}f(t)
=
T\!\left(\frac{t}{\mu},\alpha\right)f'(t).
\]
\end{proof}

\subsection{Proof of Theorem~\ref{t:mu_derivative_transform}}
\label{app:proof_mu_derivative_transform}
\begin{proof}
Since \(0<\alpha\leq 1\), the \(\mu\)-scaled generalized derivative satisfies
\[
G_{T,\mu}^{\alpha}f(t)
=
T\left(\frac{t}{\mu},\alpha\right)f'(t).
\]
Therefore,
\[
\begin{aligned}
CC_{T,\mu}^{\alpha}[G_{T,\mu}^{\alpha}f](s,\mu)
&=
\mu\int_{0}^{\infty}
e^{-s\mu x(t/\mu)}
G_{T,\mu}^{\alpha}f(t)
\frac{dt}{T(t/\mu,\alpha)}
\\
&=
\mu\int_{0}^{\infty}
e^{-s\mu x(t/\mu)}f'(t)\,dt.
\end{aligned}
\]
Now let
\[
u=e^{-s\mu x(t/\mu)},
\qquad
dv=f'(t)\,dt.
\]
Then
\[
du
=
-\frac{s}{T(t/\mu,\alpha)}
e^{-s\mu x(t/\mu)}\,dt,
\qquad
v=f(t).
\]
Using integration by parts, we obtain
\[
\begin{aligned}
CC_{T,\mu}^{\alpha}[G_{T,\mu}^{\alpha}f](s,\mu)
&=
\mu\left[
e^{-s\mu x(t/\mu)}f(t)
\right]_{0}^{\infty}
+
s\mu\int_{0}^{\infty}
e^{-s\mu x(t/\mu)}f(t)
\frac{dt}{T(t/\mu,\alpha)}.
\end{aligned}
\]
Since $\lim_{t\to\infty} e^{-s\mu x(t/\mu)}f(t)=0$, and since \(x(0)=0\), we get
$$
CC_{T,\mu}^{\alpha}[G_{T,\mu}^{\alpha}f](s,\mu)
=
s\,CC_{T,\mu}^{\alpha}[f](s,\mu)-\mu f(0).
$$
\end{proof}

\subsection{Proof of Theorem~\ref{t:mu_laplace_relation}}
\label{app:proof_mu_laplace_relation}
\begin{proof}
Since \(y_\mu(t)=\mu x(t/\mu)\), we have
\[
y_\mu'(t)=\frac{1}{T(t/\mu,\alpha)}.
\]
Thus,
\[
CC_{T,\mu}^{\alpha}[f](s,\mu)
=
\mu\int_0^\infty e^{-s y_\mu(t)}f(t)\,dy_\mu(t).
\]
Taking \(r=y_\mu(t)\), \(dr=dy_\mu(t)\), and \(t=u_\mu(r)\), we obtain
\[
CC_{T,\mu}^{\alpha}[f](s,\mu)
=
\mu\int_0^\infty e^{-sr}f(u_\mu(r))\,dr
=
\mu\int_0^\infty e^{-sr}g_\mu(r)\,dr
=
\mu L[g_\mu](s).
\]
\end{proof}

\subsection{Alternative $\mu$-convolution theorem}
\label{app:mu_convolution_alt}

The following result can be obtained in two equivalent ways. The first proof uses the relation between \(CC_{T,\mu}^{\alpha}\) and the classical Laplace transform, while the second proof derives the convolution identity directly from the definition of the \(\mu\)-scaled transform.

\begin{Theorem}
\label{t:mu_convolution_alt}
Let \(f,g:[0,\infty)\to\RR\) be measurable functions such that the following transforms exist, with \(0<\alpha\leq 1\) and \(0<\mu<\infty\). Let
\[
y_\mu(t)=\mu x\left(\frac{t}{\mu}\right),
\qquad
u_\mu=y_\mu^{-1}.
\]
Define the \(\mu\)-convolution by
$$
(f*_\mu g)(t)
=
\int_0^t
f(\tau)\,
g\!\left(u_\mu\!\left(y_\mu(t)-y_\mu(\tau)\right)\right)
\frac{d\tau}{T(\tau/\mu,\alpha)}.
$$
Then
$$
CC_{T,\mu}^{\alpha}[f*_\mu g](s,\mu)
=
\frac{1}{\mu}
CC_{T,\mu}^{\alpha}[f](s,\mu)\,
CC_{T,\mu}^{\alpha}[g](s,\mu).
$$
\end{Theorem}

\begin{proof}
Let
$$
\widehat f(r)=f(u_\mu(r)),
\qquad
\widehat g(r)=g(u_\mu(r)).
$$
Hence
$$
CC_{T,\mu}^{\alpha}[f](s,\mu)=\mu L[\widehat f](s),
\qquad
CC_{T,\mu}^{\alpha}[g](s,\mu)=\mu L[\widehat g](s).
$$
Using the classical convolution and taking \(r=y_\mu(t)\), with
\[
\xi=y_\mu(\tau),
\qquad
d\xi=\frac{d\tau}{T(\tau/\mu,\alpha)},
\]
we obtain
$$
\begin{aligned}
(\widehat f*\widehat g)(y_\mu(t))
&=
\int_0^t
\widehat f(y_\mu(\tau))\,
\widehat g(y_\mu(t)-y_\mu(\tau))
\frac{d\tau}{T(\tau/\mu,\alpha)}
\\
&=
\int_0^t
f(\tau)\,
g\!\left(u_\mu\!\left(y_\mu(t)-y_\mu(\tau)\right)\right)
\frac{d\tau}{T(\tau/\mu,\alpha)}
=
(f*_\mu g)(t).
\end{aligned}
$$
Therefore,
$$
CC_{T,\mu}^{\alpha}[f*_\mu g](s,\mu)
=
\mu L[\widehat f*\widehat g](s)
=
\mu L[\widehat f](s)L[\widehat g](s).
$$
Since
$$
L[\widehat f](s)=\frac{1}{\mu}CC_{T,\mu}^{\alpha}[f](s,\mu),
\qquad
L[\widehat g](s)=\frac{1}{\mu}CC_{T,\mu}^{\alpha}[g](s,\mu),
$$
it follows that
$$
CC_{T,\mu}^{\alpha}[f*_\mu g](s,\mu)
=
\frac{1}{\mu}
CC_{T,\mu}^{\alpha}[f](s,\mu)
CC_{T,\mu}^{\alpha}[g](s,\mu).
$$
This proves the result.
\end{proof}

\subsection{Proof of Theorem~\ref{t:mu_convolution}}
\label{app:proof_mu_convolution}
\begin{proof}
By definition and Fubini's theorem,
$$
\begin{aligned}
CC_{T,\mu}^{\alpha}[f*_\mu g](s,\mu)
&=
\mu\int_0^\infty\int_\omega^\infty
e^{-s y_\mu(t)}
f\!\left(u_\mu(y_\mu(t)-y_\mu(\omega))\right)
g(\omega)
\frac{dt}{T(t/\mu,\alpha)}
\frac{d\omega}{T(\omega/\mu,\alpha)}.
\end{aligned}
$$
Using
$$
e^{-s y_\mu(t)}
=
e^{-s y_\mu(\omega)}
e^{-s(y_\mu(t)-y_\mu(\omega))},
$$
we get
$$
\begin{aligned}
CC_{T,\mu}^{\alpha}[f*_\mu g](s,\mu)
&=
\mu\int_0^\infty
e^{-s y_\mu(\omega)}g(\omega)
\frac{d\omega}{T(\omega/\mu,\alpha)}
\\
&\quad\times
\int_\omega^\infty
e^{-s(y_\mu(t)-y_\mu(\omega))}
f\!\left(u_\mu(y_\mu(t)-y_\mu(\omega))\right)
\frac{dt}{T(t/\mu,\alpha)}.
\end{aligned}
$$
Now set
$$
y_\mu(z)=y_\mu(t)-y_\mu(\omega).
$$
Then
$$
\frac{dt}{T(t/\mu,\alpha)}
=
dy_\mu(t)
=
dy_\mu(z)
=
\frac{dz}{T(z/\mu,\alpha)},
\qquad
u_\mu(y_\mu(t)-y_\mu(\omega))=z.
$$
Hence,
$$
\begin{aligned}
CC_{T,\mu}^{\alpha}[f*_\mu g](s,\mu)
&=
\mu
\left(
\int_0^\infty
e^{-s y_\mu(z)}f(z)
\frac{dz}{T(z/\mu,\alpha)}
\right)
\left(
\int_0^\infty
e^{-s y_\mu(\omega)}g(\omega)
\frac{d\omega}{T(\omega/\mu,\alpha)}
\right)
\\
&=
\mu
\left(
\frac{1}{\mu}CC_{T,\mu}^{\alpha}[f](s,\mu)
\right)
\left(
\frac{1}{\mu}CC_{T,\mu}^{\alpha}[g](s,\mu)
\right).
\end{aligned}
$$
Therefore,
$$
CC_{T,\mu}^{\alpha}[f*_\mu g](s,\mu)
=
\frac{1}{\mu}
CC_{T,\mu}^{\alpha}[f](s,\mu)
CC_{T,\mu}^{\alpha}[g](s,\mu).
$$
This proves the result.
\end{proof}

\subsection{Proof of Proposition~\ref{constant}}
\label{app:proof_constant}
\begin{proof}
By definition,
$$
CC_{T,\mu}^{\alpha}[1](s,\mu)
=
\mu\int_{0}^{\infty}
e^{-s\mu x(t/\mu)}
\frac{dt}{T(t/\mu,\alpha)}.
$$
Since
$$
d\,x(t/\mu)=\frac{dt}{\mu T(t/\mu,\alpha)},
\qquad
\frac{dt}{T(t/\mu,\alpha)}
=
\mu\,d\,x(t/\mu),
$$
we obtain
$$
CC_{T,\mu}^{\alpha}[1](s,\mu)
=
\mu^2\int_{0}^{\infty}
e^{-s\mu x(t/\mu)}\,d\,x(t/\mu)
=
\mu^2
\left(
\frac{1}{-s\mu}
\left[
e^{-s\mu x(t/\mu)}
\right]_{0}^{\infty}
\right).
$$
Using \(s>0\), \(x(0)=0\), and \(x(t/\mu)\to\infty\) as \(t\to\infty\), it follows that
$$
CC_{T,\mu}^{\alpha}[1](s,\mu)
=
\frac{\mu^2}{s\mu}
=
\frac{\mu}{s}.
$$
\end{proof}

\subsection{Proof of Proposition~\ref{exp}}
\label{app:proof_exp}
\begin{proof}
Since
$$
E_\alpha(\mu c,t/\mu)
=
e^{c y_\mu(t)},
\qquad
y_\mu(t)=\mu x(t/\mu),
$$
we have
$$
\begin{aligned}
CC_{T,\mu}^{\alpha}\left[E_\alpha(\mu c,t/\mu)\right](s,\mu)
&=
\mu\int_0^\infty
e^{-s y_\mu(t)}e^{c y_\mu(t)}
\frac{dt}{T(t/\mu,\alpha)}
\\
&=
\mu\int_0^\infty e^{-(s-c)y_\mu(t)}\,dy_\mu(t).
\end{aligned}
$$
It follows that
$$
CC_{T,\mu}^{\alpha}\left[E_\alpha(\mu c,t/\mu)\right](s,\mu)
=
\mu\left[
\frac{-1}{s-c}e^{-(s-c)y_\mu(t)}
\right]_{0}^{\infty}
=
\frac{\mu}{s-c}.
$$
\end{proof}

\subsection{Proof of Corollary~\ref{invrelat}}
\label{app:proof_invrelat}
\begin{proof}
By Theorem~\ref{t:mu_laplace_relation},
$$
F(s,\mu)
=
CC_{T,\mu}^{\alpha}[f](s,\mu)
=
\mu L[g_\mu](s).
$$
Hence,
$$
L[g_\mu](s)=\frac{1}{\mu}F(s,\mu).
$$
Applying the classical inverse Laplace transform gives
$$
g_\mu(r)
=
L^{-1}\left[
\frac{1}{\mu}F(s,\mu)
\right](r).
$$
Since \(g_\mu(r)=f(u_\mu(r))\), taking \(r=y_\mu(t)\) gives
$$
f(t)
=
f(u_\mu(y_\mu(t)))
=
L^{-1}\left[
\frac{1}{\mu}F(s,\mu)
\right](y_\mu(t)).
$$
Finally, using
$$
y_\mu(t)=\mu x\left(\frac{t}{\mu}\right),
$$
from which it follows
$$
f(t)
=
L^{-1}\left[
\frac{1}{\mu}CC_{T,\mu}^{\alpha}[f](s,\mu)
\right]
\left(
\mu x\left(\frac{t}{\mu}\right)
\right).
$$
\end{proof}

\subsection{Computation of the first correction component}
\label{app:computation_N1_mu}

The first correction component \(N_1(t,\mu)\) satisfies
$$
CC_{T,\mu}^{\alpha}[N_1](s,\mu)
=
\frac{k_2}{s-k_1}CC_{T,\mu}^{\alpha}[A_0](s,\mu)
-
\frac{k_3}{s-k_1}CC_{T,\mu}^{\alpha}[B_0](s,\mu).
$$
Using Theorem~\ref{t:mu_convolution}, we obtain
$$
N_1(t,\mu)
=
k_2\left(E_{\alpha}(\mu k_1,\cdot/\mu)*_{\mu}A_0\right)(t)
-
k_3\left(E_{\alpha}(\mu k_1,\cdot/\mu)*_{\mu}B_0\right)(t).
$$
Thus,
$$
\begin{aligned}
N_1(t,\mu)
&=
k_2\int_0^t
E_{\alpha}\!\left(
\mu k_1,
\frac{u_{\mu}(y_{\mu}(t)-y_{\mu}(\omega))}{\mu}
\right)
\frac{1}{E_{\alpha}(\mu k_1,\omega/\mu)h_1+h_3}
\frac{d\omega}{T(\omega/\mu,\alpha)}
\\
&\quad
-
k_3\int_0^t
E_{\alpha}\!\left(
\mu k_1,
\frac{u_{\mu}(y_{\mu}(t)-y_{\mu}(\omega))}{\mu}
\right)
\left(E_{\alpha}(\mu k_1,\omega/\mu)h_1+h_2\right)^2
\frac{d\omega}{T(\omega/\mu,\alpha)}.
\end{aligned}
$$

Since
$$
E_{\alpha}\!\left(
\mu k_1,
\frac{u_{\mu}(y_{\mu}(t)-y_{\mu}(\omega))}{\mu}
\right)
=
\exp\!\left(k_1(y_\mu(t)-y_\mu(\omega))\right)
=
E_{\alpha}(\mu k_1,t/\mu)E_{\alpha}(-\mu k_1,\omega/\mu),
$$
we can factor out the term depending only on \(t\):
$$
N_1(t,\mu)
=
E_{\alpha}(\mu k_1,t/\mu)
\left[
k_2 I_1(t,\mu)-k_3 I_2(t,\mu)
\right],
$$
where
$$
I_1(t,\mu)
=
\int_0^t
\frac{E_{\alpha}(-\mu k_1,\omega/\mu)}
{E_{\alpha}(\mu k_1,\omega/\mu)h_1+h_3}
\frac{d\omega}{T(\omega/\mu,\alpha)}
$$
and
$$
I_2(t,\mu)
=
\int_0^t
E_{\alpha}(-\mu k_1,\omega/\mu)
\left(E_{\alpha}(\mu k_1,\omega/\mu)h_1+h_2\right)^2
\frac{d\omega}{T(\omega/\mu,\alpha)}.
$$

Let
$$
E=E(\omega):=E_{\alpha}(\mu k_1,\omega/\mu).
$$
Then
$$
E_{\alpha}(-\mu k_1,\omega/\mu)=\frac{1}{E},
\qquad
\frac{dE}{d\omega}
=
\frac{k_1E}{T(\omega/\mu,\alpha)},
\qquad
\frac{d\omega}{T(\omega/\mu,\alpha)}=\frac{dE}{k_1E}.
$$
Also,
$$
E(0)=1,
\qquad
E(t)=E_{\alpha}(\mu k_1,t/\mu).
$$

For \(I_1\), we get
$$
I_1(t,\mu)
=
\frac{1}{k_1}
\int_1^{E(t)}
\frac{dE}{E^2(h_1E+h_3)}.
$$
Using
$$
\frac{1}{E^2(h_1E+h_3)}
=
-\frac{h_1}{h_3^2}\frac{1}{E}
+
\frac{1}{h_3}\frac{1}{E^2}
+
\frac{h_1^2}{h_3^2}\frac{1}{h_1E+h_3},
$$
therefore
$$
I_1(t,\mu)
=
\frac{1}{k_1h_3}\left(1-\frac{1}{E(t)}\right)
+
\frac{h_1}{k_1h_3^2}
\ln\!\left(
\frac{h_1E(t)+h_3}{(h_1+h_3)E(t)}
\right).
$$

Similarly,
$$
I_2(t,\mu)
=
\frac{1}{k_1}
\int_1^{E(t)}
\frac{(h_1E+h_2)^2}{E^2}\,dE.
$$
Since
$$
\frac{(h_1E+h_2)^2}{E^2}
=
h_1^2+\frac{2h_1h_2}{E}+\frac{h_2^2}{E^2},
$$
we get
$$
I_2(t,\mu)
=
\frac{h_1^2}{k_1}\bigl(E(t)-1\bigr)
+
\frac{2h_1h_2}{k_1}\ln E(t)
+
\frac{h_2^2}{k_1}\left(1-\frac{1}{E(t)}\right).
$$
Using \(\ln E(t)=\mu k_1x(t/\mu)\), this becomes
$$
I_2(t,\mu)
=
\frac{h_1^2}{k_1}\bigl(E(t)-1\bigr)
+
2\mu h_1h_2\,x(t/\mu)
+
\frac{h_2^2}{k_1}\left(1-\frac{1}{E(t)}\right).
$$

Therefore,
$$
N_1(t,\mu)
=
E_{\alpha}(\mu k_1,t/\mu)
\left[
k_2I_1(t,\mu)-k_3I_2(t,\mu)
\right],
$$
where
$$
I_1(t,\mu)
=
\frac{1}{k_1h_3}\left(1-E_{\alpha}(-\mu k_1,t/\mu)\right)
+
\frac{h_1}{k_1h_3^2}
\ln\!\left(
\frac{h_1E_{\alpha}(\mu k_1,t/\mu)+h_3}
{(h_1+h_3)E_{\alpha}(\mu k_1,t/\mu)}
\right),
$$
and
$$
I_2(t,\mu)
=
\frac{h_1^2}{k_1}\left(E_{\alpha}(\mu k_1,t/\mu)-1\right)
+
2\mu h_1h_2\,x(t/\mu)
+
\frac{h_2^2}{k_1}
\left(1-E_{\alpha}(-\mu k_1,t/\mu)\right).
$$

\subsection{Proof of Theorem~\ref{t:inverse_mu}}
\label{app:proof_inverse_mu}
\begin{proof}
Let
\[
y_\mu(t)=\mu x\left(\frac{t}{\mu}\right),
\]
and let \(u_\mu\) denote the inverse function of \(y_\mu\). By the
change of variables
\[
r=y_\mu(t),
\]
we have
\[
\frac{dr}{dt}
=
\frac{1}{T(t/\mu,\alpha)}.
\]
Therefore, from the definition of the \(\mu\)-scaled generalized
transform,
\[
CC_{T,\mu}^{\alpha}[f](s,\mu)
=
\mu\int_0^\infty
e^{-s y_\mu(t)}f(t)
\frac{dt}{T(t/\mu,\alpha)},
\]
we obtain
\[
CC_{T,\mu}^{\alpha}[f](s,\mu)
=
\mu\int_0^\infty e^{-sr}f(u_\mu(r))\,dr.
\]
Hence,
\[
CC_{T,\mu}^{\alpha}[f](s,\mu)
=
\mu L[f(u_\mu(r))](s).
\]
Equivalently,
\[
L[f(u_\mu(r))](s)
=
\frac{1}{\mu}CC_{T,\mu}^{\alpha}[f](s,\mu).
\]
Applying the inverse Laplace formula gives
\[
f(u_\mu(r))
=
\frac{1}{2\pi i}
\lim_{R\to\infty}
\int_{a-iR}^{a+iR}
e^{sr}
\frac{1}{\mu}
CC_{T,\mu}^{\alpha}[f](s,\mu)\,ds.
\]
Finally, taking \(r=y_\mu(t)\), we get
\[
\begin{aligned}
f(t)
&=
f(u_\mu(y_\mu(t)))\\
&=
\frac{1}{2\pi i\,\mu}
\lim_{R\to\infty}
\int_{a-iR}^{a+iR}
e^{s y_\mu(t)}
CC_{T,\mu}^{\alpha}[f](s,\mu)\,ds.
\end{aligned}
\]
This proves the result.
\end{proof}

\subsection{Proof of Theorem~\ref{t:residue_generalized_transform}}
\label{app:proof_residue_generalized_transform}
\begin{proof}
Let
\[
r=x(t),
\qquad
x(t)=\int_0^t\frac{d\omega}{T(\omega,\alpha)},
\]
and let \(u\) be the inverse function of \(x\). Define
\[
g(r)=f(u(r)).
\]
By Corollary~\ref{invL}, we have
\[
\mathcal{L}_{T}^{\alpha}[f](s)
=
L[g](s).
\]
Since \(F(s)=\mathcal{L}_{T}^{\alpha}[f](s)\), it follows that
\[
F(s)=L[g](s).
\]
Then,
\[
g(r)
=
\frac{1}{2\pi i}
\lim_{R\to\infty}
\int_{a-iR}^{a+iR}
e^{sr}F(s)\,ds
=
\sum_{k=1}^{n}
\operatorname{Res}
\left(
e^{sr}F(s),s_k
\right).
\]
Taking \(r=x(t)\), hence
\[
g(x(t))=f(u(x(t)))=f(t).
\]
Therefore,
\[
f(t)
=
\frac{1}{2\pi i}
\lim_{R\to\infty}
\int_{a-iR}^{a+iR}
e^{s x(t)}F(s)\,ds
=
\sum_{k=1}^{n}
\operatorname{Res}
\left(
e^{s x(t)}F(s),s_k
\right).
\]
Finally, since
\[
e^{s x(t)}
=
\exp\left(
s\int_0^t\frac{d\omega}{T(\omega,\alpha)}
\right)
=
E_{\alpha}(s,t),
\]
we conclude that
\[
f(t)
=
\frac{1}{2\pi i}
\lim_{R\to\infty}
\int_{a-iR}^{a+iR}
E_{\alpha}(s,t)F(s)\,ds
=
\sum_{k=1}^{n}
\operatorname{Res}
\left(
E_{\alpha}(s,t)F(s),s_k
\right).
\]
The proof is complete.
\end{proof}

\subsection{Proof of Theorem~\ref{t:residue_mu_transform}}
\label{app:proof_residue_mu_transform}
\begin{proof}
Let
\[
r=y_\mu(t),
\qquad
y_\mu(t)=\mu x\left(\frac{t}{\mu}\right),
\]
and let \(u_\mu=y_\mu^{-1}\) be the inverse function of \(y_\mu\). Define
\[
g_\mu(r)=f(u_\mu(r)).
\]
By Corollary~\ref{invrelat}, we get
\[
CC_{T,\mu}^{\alpha}[f](s,\mu)
=
\mu L[g_\mu](s).
\]
Since \(F(s,\mu)=CC_{T,\mu}^{\alpha}[f](s,\mu)\), it follows that
\[
L[g_\mu](s)
=
\frac{1}{\mu}F(s,\mu).
\]
Then,
\[
g_\mu(r)
=
\frac{1}{2\pi i}
\lim_{R\to\infty}
\int_{a-iR}^{a+iR}
e^{sr}\frac{1}{\mu}F(s,\mu)\,ds.
\]
Hence,
\[
g_\mu(r)
=
\frac{1}{2\pi i\,\mu}
\lim_{R\to\infty}
\int_{a-iR}^{a+iR}
e^{sr}F(s,\mu)\,ds
=
\frac{1}{\mu}
\sum_{k=1}^{n}
\operatorname{Res}
\left(
e^{sr}F(s,\mu),s_k
\right).
\]
Taking \(r=y_\mu(t)\), we obtain
\[
g_\mu(y_\mu(t))=f(u_\mu(y_\mu(t)))=f(t).
\]
Therefore,
\[
f(t)
=
\frac{1}{2\pi i\,\mu}
\lim_{R\to\infty}
\int_{a-iR}^{a+iR}
e^{s y_\mu(t)}F(s,\mu)\,ds
=
\frac{1}{\mu}
\sum_{k=1}^{n}
\operatorname{Res}
\left(
e^{s y_\mu(t)}F(s,\mu),s_k
\right).
\]
Hence the identity holds.
\end{proof}

\subsection{Proof of Theorem~\ref{t:mu_exponential_order}}
\label{app:proof_mu_exponential_order}

\begin{proof}
Let
\[
y_\mu(t)=\mu x\left(\frac{t}{\mu}\right).
\]
Then
\[
dy_\mu(t)=\frac{dt}{T(t/\mu,\alpha)}
\]
and
\[
E_{\alpha}(\mu c,t/\mu)
=
\exp\left(
\mu c\int_0^{t/\mu}\frac{d\omega}{T(\omega,\alpha)}
\right)
=
e^{c y_\mu(t)}.
\]

By hypothesis, there exist constants \(M>0\), \(c\in\RR\), and \(t_0\geq0\) such that
\[
|f(t)|
\leq
M E_{\alpha}(\mu c,t/\mu)
=
M e^{c y_\mu(t)},
\qquad t\geq t_0.
\]

We split the transform integral as
\[
\begin{aligned}
CC^\alpha_{T,\mu}[f](s,\mu)
&=
\mu\int_0^{t_0}
e^{-s y_\mu(t)}f(t)\frac{dt}{T(t/\mu,\alpha)}
\\
&\quad+
\mu\int_{t_0}^{\infty}
e^{-s y_\mu(t)}f(t)\frac{dt}{T(t/\mu,\alpha)}.
\end{aligned}
\]

The first integral is finite because \(f\) is locally integrable and the measure
\[
\frac{dt}{T(t/\mu,\alpha)}
\]
has finite mass on bounded intervals under the standing assumptions on \(T\). Indeed, after the change of variable
\[
\omega=\frac{t}{\mu},
\]
we have
\[
\int_0^{t_0}\frac{dt}{T(t/\mu,\alpha)}
=
\mu\int_0^{t_0/\mu}\frac{d\omega}{T(\omega,\alpha)}
<\infty.
\]
Moreover, \(e^{-s y_\mu(t)}\) is bounded on \([0,t_0]\), since \(y_\mu\) is finite there. Hence the local part of the integral is well defined.

Now let \(\Re(s)>c\). For \(t\geq t_0\), we have
\[
\left|e^{-s y_\mu(t)}f(t)\right|
\leq
e^{-\Re(s)y_\mu(t)}|f(t)|
\leq
M e^{-(\Re(s)-c)y_\mu(t)}.
\]
Since \(\Re(s)-c>0\) and \(y_\mu(t)\to\infty\) as \(t\to\infty\), it follows that
\[
\lim_{t\to\infty}
e^{-s y_\mu(t)}f(t)=0.
\]

Furthermore,
\[
\begin{aligned}
\mu\int_{t_0}^{\infty}
\left|e^{-s y_\mu(t)}f(t)\right|
\frac{dt}{T(t/\mu,\alpha)}
&\leq
\mu M\int_{t_0}^{\infty}
e^{-(\Re(s)-c)y_\mu(t)}
\frac{dt}{T(t/\mu,\alpha)}
\\
&=
\mu M\int_{y_\mu(t_0)}^{\infty}
e^{-(\Re(s)-c)r}\,dr
\\
&=
\frac{\mu M}{\Re(s)-c}
e^{-(\Re(s)-c)y_\mu(t_0)}
<\infty.
\end{aligned}
\]
Therefore, \(CC_{T,\mu}^{\alpha}[f](s,\mu)\) exists for every
\(s\in\mathbb{C}\) such that \(\Re(s)>c\).

If, in addition, the estimate
\[
|f(t)|\leq M E_{\alpha}(\mu c,t/\mu)
\]
holds for every \(t\geq0\), then
\[
\begin{aligned}
\left|
CC_{T,\mu}^{\alpha}[f](s,\mu)
\right|
&\leq
\mu\int_0^\infty
\left|e^{-s y_\mu(t)}f(t)\right|
\frac{dt}{T(t/\mu,\alpha)}
\\
&\leq
\mu M\int_0^\infty
e^{-(\Re(s)-c)y_\mu(t)}
\frac{dt}{T(t/\mu,\alpha)}
\\
&=
\mu M\int_0^\infty
e^{-(\Re(s)-c)r}\,dr
\\
&=
\frac{\mu M}{\Re(s)-c}.
\end{aligned}
\]
This proves the desired bound.
\end{proof}


\end{document}